\documentclass[11pt]{amsart}

\usepackage{fullpage}

\usepackage{amsmath, amscd, amsthm, amssymb}

\usepackage{hyperref}

\def\C{{\mathbf C}}
\def\R{{\mathbf R}}
\def\Z{{\mathbf Z}}
\def\Q{{\mathbf Q}}
\def\A{{\mathbf A}}

\newtheorem{theorem}{Theorem}[section]

\newtheorem{lemma}[theorem]{Lemma}

\newtheorem{proposition}[theorem]{Proposition}

\newtheorem{corollary}[theorem]{Corollary}

\newtheorem{claim}[theorem]{Claim}

\theoremstyle{definition}
\newtheorem{definition}[theorem]{Definition}

\theoremstyle{remark}
\newtheorem{remark}[theorem]{Remark}

\newcommand{\mm}[4]{\left(\begin{smallmatrix} #1 & #2\\ #3 & #4\end{smallmatrix}\right)}
\newcommand{\mb}[4]{\left(\begin{array}{cc} #1 & #2\\ #3 & #4\end{array}\right)}

\DeclareMathOperator{\val}{val}
\DeclareMathOperator{\tr}{tr}

\DeclareMathOperator{\SO}{SO}
\DeclareMathOperator{\Spin}{Spin}

\DeclareMathOperator{\Sp}{Sp}
\DeclareMathOperator{\GSp}{GSp}
\DeclareMathOperator{\PGSp}{PGSp}

\DeclareMathOperator{\SU}{SU}

\DeclareMathOperator{\SL}{SL}
\DeclareMathOperator{\GL}{GL}

\DeclareMathOperator{\charf}{char}
\DeclareMathOperator{\diag}{diag}
\DeclareMathOperator{\Hom}{Hom}

\def\g{{\mathfrak g}}
\def\h{{\mathfrak h}}
\def\k{{\mathfrak k}}
\def\p{{\mathfrak p}}
\def\sl{{\mathfrak sl}}

\def\n{{\mathfrak n}}

\def\V{{\mathbb V}}
\def\W{{\mathcal W}}

\begin{document}
\title{Modular forms on $G_2$ and their standard $L$-function}
\author{Aaron Pollack}
\address{Department of Mathematics\\ Institute for Advanced Study\\ Princeton, NJ USA}
\email{aaronjp@math.ias.edu}
\address{Department of Mathematics\\ Duke University\\ Durham, NC USA}
\email{apollack@math.duke.edu}
\thanks{This work supported by the Schmidt fund at IAS.  We thank the IAS for its hospitality and for providing an excellent working environment.  We also thank the Simons Foundation for its support of the symposium on the Relative Trace Formula, where some aspects of this work were discussed}

\begin{abstract} The purpose of this partly expository paper is to give an introduction to modular forms on $G_2$.  We do this by focusing on two aspects of $G_2$ modular forms.  First, we discuss the Fourier expansion of modular forms, following work of Gan-Gross-Savin and the author.  Then, following Gurevich-Segal and Segal, we discuss a Rankin-Selberg integral yielding the standard $L$-function of modular forms on $G_2$.  As a corollary of the analysis of this Rankin-Selberg integral, one obtains a Dirichlet series for the standard $L$-function of $G_2$ modular forms; this involves the arithmetic invariant theory of cubic rings.  We end by analyzing the archimedean zeta integral that arises from the Rankin-Selberg integral when the cusp form is an even weight modular form.\end{abstract}
\maketitle

\section{Introduction}  If a reductive group $G$ over $\Q$ has a Hermitian symmetric space $G(\R)/K$, then there is a notion of \emph{modular forms} on $G$.  Namely, one can consider the automorphic forms on $G$ that give rise to holomorphic functions on $G(\R)/K$.  However, if $G$ does not have a Hermitian symmetric space, then $G$ has no obvious notion of modular forms, or very special automorphic forms.  The split exceptional group $G_2$ is an example: its symmetric space $G_2(\R)/\left(\SU(2) \times\SU(2)\right)$ has no complex structure.  Nevertheless, Gan-Gross-Savin \cite{ganGrossSavin} defined and studied modular forms on $G_2$ by building off of work of Gross-Wallach \cite{grossWallach2} on quaternionic discrete series representations.

The purpose of this paper is to give an introduction to modular forms on $G_2$, with the hope that we might interest others in this barely-developed area of mathematics.  Indeed, very little is known about $G_2$-modular forms.  For example, as pointed out in \cite{ganGrossSavin}, one does not even know the smallest positive integer $k$ for which there is a nonzero level-one cuspidal modular form on $G_2$ of weight $k$.

Besides recalling for the reader some basic facts about the group $G_2$ and the definition of its modular forms, we focus on two aspects of $G_2$-modular forms.  The first thing we focus on is their Fourier expansion.  Using a certain multiplicity one result of Wallach \cite{wallach} on ``generalized Whittaker vectors'', Gan-Gross-Savin defined the (generic) Fourier coefficients of a $G_2$ modular form.  This involves Wallach's multiplicity one result and the arithmetic invariant theory of cubic rings.  Using their definition of Fourier coefficients, Gan-Gross-Savin computed the Fourier coefficients of an explicit theta function and (using \cite{jiangRallis}) some Eisenstein series, and also computed how Hecke operators act on these Fourier coefficients.  

The paper \cite{pollackQDS} makes explicit the Fourier \emph{expansion} of modular forms, i.e., it involves functions on $G_2(\R)$ and not just the Fourier coefficients, which are complex numbers.  Furthermore, as a consequence of \emph{loc cit} one can also define the non-generic Fourier coefficients of $G_2$-modular forms.  The first result we discuss is this Fourier expansion of $G_2$-modular forms.  This is the $G_2$-special case of the main result of \cite{pollackQDS}.  Because the constructions and arguments in \emph{loc cit} cover all the exceptional Dynkin types, they rely on sometimes lengthy Jordan algebra computations.  For $G_2$ one does not need Jordan algebras, so we give a treatment that is simpler, and that we expect is easier for the non-expert to follow.

The second aspect of $G_2$-modular forms that we discuss is their standard $L$-function.  Here we rely on the exceptional papers \cite{gurevichSegal} and \cite{segal} of Gurevich-Segal and Segal.  These works give a Rankin-Selberg integral yielding the degree $7$ standard $L$-function of automorphic cusp forms on $G_2$.  Moreover, the integrals of \emph{loc cit} are non-vanishing for $G_2$-modular forms, and in fact rely on the same type of Fourier coefficients mentioned above.  (The older Rankin-Selberg integral \cite{ginzburgG2} of Ginzburg relies on a Whittaker coefficient and vanishes identically for $G_2$-modular forms.)

The beautiful results of the papers \cite{segal} and \cite{gurevichSegal} have somewhat involved proofs.  We give a streamlined account of the integrals of \emph{loc cit} by simplifying a bit some of the computations.  We then use this analysis of the Rankin-Selberg integral to a give a Dirichlet series for the standard $L$-function of modular forms on $G_2$.  The Dirichlet series has a nice expression in terms of cubic rings and the Fourier coefficients of modular forms.

The Dirichlet series comes from analyzing the local zeta integrals at the good finite primes that arise from the global Rankin-Selberg convolution.  One can also ask for an analysis of the archimedean zeta integral, when the cusp form involved is a modular form.  Given our computation of the Fourier expansion of modular forms, this question becomes tractable.  In the final section, we analyze this archimedean zeta integral and relate it to products of gamma functions.

We now give an outline of the rest of the paper.  In section \ref{sec:G2andOct}, we review the octonions and basic facts about the group $G_2$.  Next, in section \ref{sec:MFs}, we recall the definition from \cite{ganGrossSavin} of modular forms on $G_2$ and their Fourier coefficients.  We also state the results about the Fourier expansion (Theorem \ref{thm:FE}) and Dirichlet series (Theorem \ref{thm:DS}) of $G_2$-modular forms.  In section \ref{sec:genWhit}, we prove Theorem \ref{thm:FE}, after \cite{pollackQDS}.  In section \ref{sec:RS} we discuss the Rankin-Selberg integral of \cite{gurevichSegal} and \cite{segal}, and also the Dirichlet series for the standard $L$-function.  In particular, we prove Theorem \ref{thm:DS}.  Finally, in section \ref{sec:arch}, we apply our results on the Fourier expansion of modular forms to analyze the archimedean zeta integral that arises from the Rankin-Selberg integral of \emph{loc cit} when the cusp form is a modular form.

\section{Octonions and the group $G_2$}\label{sec:G2andOct} In this section, we define the group $G_2$ as the automorphisms of the split octonions, and review basic structural information about $G_2$: its Lie algebra, Heisenberg parabolic subgroup, Cartan involution, and maximal compact subgroup.

\subsection{The octonions} In this subsection we very briefly review the octonions, following the discussion in \cite[Section 1.1]{pollackWanZydor}.  The reader might see \cite{springerVeldkampBook} for a more thorough treatment.

Thus suppose $F$ is field of characteristic $0$.  We denote by $\Theta$ the split octonions over $F$.  Let $V_3$ denote the defining representation of the group $\SL_3$, and $V_3^\vee$ the dual representation. As a vector space,
\[\Theta = \left\{\mb{a}{v}{\phi}{d}: a,d \in F, v \in V_3, \phi \in V_3^\vee\right\}.\]
This is the so-called Zorn model of the octonions.  

As an $F$-algebra, $\Theta$ is non-commutative and non-associative.  Because $V_3$ is considered as a representation of $\SL_3$, we assume given an identification $\wedge^3 V_3 \simeq F$, which induces $\wedge^2 V_3 \simeq V_3^\vee$ and $\wedge^2 V_3^\vee \simeq V_3$.  The multiplication in $\Theta$ is given as
\[\left(\begin{array}{cc} a & v \\ \phi & d \end{array}\right) \left(\begin{array}{cc} a' & v' \\ \phi' & d'\end{array}\right) = \left(\begin{array}{cc} aa'+ \phi'(v) & av' +d'v - \phi \wedge \phi' \\ a'\phi + d\phi'+ v \wedge v' & \phi(v') + dd'\end{array}\right).\]

The octonions also come equipped with a quadratic norm $N: \Theta \rightarrow F$ defined by $N\mm{a}{v}{\phi}{d} = ad - \phi(v)$. If $x, x' \in \Theta$, then $N(xx') = N(x)N(x')$.  Associated to the norm $N$ is a symmetric bilinear form defined by $(x,y) = N(x+y) - N(x) - N(y)$. For $x = \left(\begin{array}{cc} a & v \\ \phi & d \end{array}\right)$ and $x'=\left(\begin{array}{cc} a' & v' \\ \phi' & d'\end{array}\right)$, one has
\[(x,x') = ad' + a'd - \phi(v') - \phi'(v).\]
This bilinear form is non-degenerate, and is split.  I.e., there are isotropic subspaces of $\Theta$ of dimension four. 

Like a ring of quaternions over $F$, the octonions come equipped with a conjugation $*:\Theta \rightarrow \Theta$, which is an order-reversing involution on $\Theta$.  The conjugate $x^*$ of $x = \mb{a}{v}{\phi}{d}$ is $x^* = \mb{d}{-v}{-\phi}{a}$.  One has $x x^* = N(x) 1$ inside $\Theta$.  Define the trace of $x \in\Theta$ as $\tr(\mm{a}{v}{\phi}{d}) = a+d$.  For $x \in \Theta$, one has the quadratic equation $x^2 = \tr(x)x-N(x)$.

Write $V_7$ for the elements of $\Theta$ with trace $0$.  Thus $\mm{a}{v}{\phi}{d} \in V_7$ if and only if $a+d = 0$.  The trace $0$ elements are the perpendicular space to the line generated by the element $1 \in \Theta$.  For $x \in \Theta$, we write $Im(x) = \frac{1}{2}(x - x^*)$, so that $x = \tr(x)/2 + Im(x)$.  This is the orthogonal decomposition of $x$ into $F \cdot 1 \oplus V_7$.

The split algebraic group $G_2$ over $F$ is defined to be the automorphisms of $\Theta$: 
\[G_2(F) = \{g \in \GL(\Theta): g (x \cdot y) = gx \cdot gy \text{ for all } x, y \in \Theta\}.\]
One has $g 1 = 1$, and $g(x^*) = (gx)^*$.  Since $G_2$ stabilizes the element $1$, it acts on $V_7$, preserving the induced quadratic form.  This representation $G_2 \rightarrow \GL(V_7)$ factors through $\SO(V_7)$.

The group $G_2$ has two conjugacy classes of maximal parabolic subgroups.  Their associated flag varieties can be neatly described in terms of subspaces of $\Theta$ with particular properties under the multiplication on $\Theta$. See, for example, \cite[Proposition 1.3]{pollackWanZydor}.

\subsection{The Lie algebra $\mathfrak{g}_2$.} One can write down the Lie algebra of $G_2$ relatively easily inside the Lie algebra of $\SO(V_7)$.  To do this, first recall that if $V$ is a non-degenerate quadratic space, then $\mathfrak{so}(V) \simeq \wedge^2 V$.  In this identification, an element $w \wedge x$ acts (on the left) of $V$ as
\[(w\wedge x) \cdot v = (x,v) w - (w,v) x.\]
The Lie bracket in this notation is
\begin{equation}\label{bracket}[w \wedge x, y \wedge z] = (x,y) w\wedge z - (x,z) w \wedge y - (w,y) x \wedge z + (w,z) x \wedge y.\end{equation}

The multiplication on $\Theta$ defines a map $V_7 \otimes V_7 \rightarrow V_7$ via $w \otimes x \mapsto Im(wx)$. This map is alternating, i.e., $Im(xw) = -Im(wx)$ for $x,w \in V_7$, and thus induces a map $\wedge^2 V_7 \rightarrow V_7$. This map is surjective and $G_2$-equivariant. Denote by $\g$ the kernel of $\wedge^2 V_7 \rightarrow V_7$.  Then $\g$ is closed under the Lie bracket induced from (\ref{bracket}), and is the Lie algebra of $G_2$.

We now give an explicit basis for $\g$, and make some special notation which we will use throughout.  First, denote by $e_1, e_2, e_3$ a fixed basis of $V_3$, and write $e_1^*, e_2^*, e_3^*$ for the basis of $V_3^\vee$ dual to the $e_i$.  Set $u_0 = \mm{1}{}{}{-1} \in V_7$.  We will abuse notation and also let $e_i, e_j^*$ denote elements in $V_7$.  Thus $(e_i,e_j^*) = - \delta_{ij}$, $(u_0,u_0) = -2$, and $(u_0,e_i) = (u_0,e_j^*) = 0$ for all $i,j$.

We set $E_{kj} = e_j^* \wedge e_k$, $v_j = u_0 \wedge e_j + e_{j+1}^* \wedge e_{j+2}^*$, and $\delta_j = u_0 \wedge e_j^* + e_{j+1} \wedge e_{j+2}$ (indices taken modulo three).  One checks immediately from the definition of multiplication in $\Theta$ that the elements $v_j$ and $\delta_j$ are in the kernel of $\wedge^2 V_7 \rightarrow V_7$, and thus in $\g$.  The same goes for $E_{kj}$ so long as $j \neq k$.  A sum $\alpha_1 E_{11} + \alpha_2 E_{22} + \alpha_3 E_{33}$ is in $\g$ if and only if $\alpha_1 +\alpha_2 +\alpha_3 = 0$.

The above elements span $\g$.  We write 
\[\h = \{\alpha_1 E_{11} + \alpha_2 E_{22} + \alpha_3 E_{33}: \alpha_1 +\alpha_2 +\alpha_3 = 0\}.\]
This is a Cartan subalgebra of $\g$.  The Cartan subalgebra $\h$ acts on $E_{kj}$ by $\alpha_k - \alpha_j$, i.e.,
\[[\sum_{i}{\alpha_i E_{ii}}, E_{kj}] = (\alpha_k-\alpha_j) E_{kj}.\]
These are the long roots for $\h$.  Together with $\h$, the $E_{jk}$ span the Lie algebra $\sl_3$.  The Cartan $\h$ acts on $v_j$ via $\alpha_j$ and $\delta_j$ via $-\alpha_j$.  These are the short roots.

One has the following Lie bracket relations, which can be checked easily.  
\begin{itemize}
\item $[\delta_{j-1},v_j] = 3 E_{j,j-1}$
\item $[v_{j-1},\delta_j] = -3 E_{j-1,j}$
\item $[\delta_{j-1},\delta_j] = 2 v_{j+1}$
\item $[v_{j-1},v_j] = 2 \delta_{j+1}$
\item $[\delta_j,v_j] = 3E_{jj}- (E_{11}+E_{22}+E_{33})$.\end{itemize}
All indices here are taken modulo $3$.  We will choose a positive system on $\g$ by letting $E_{12}$ and $v_2$ be the positive simple roots.

Abstractly, $\g=\g_2 = \sl_3 \oplus V_3 \oplus V_3^\vee$, and this is a $\Z/3$-grading. See, e.g., \cite[chapter 22]{fultonHarris}. In fact, all the (split) exceptional Lie algebras have $\Z/3$-grading that generalizes this one.  See \cite{rumelhart} or \cite{savin}.  In this $\Z/3$-model, the Lie bracket is given as follows:
\begin{enumerate}
\item the commutators $[\sl_3, V_3]$ and $[\sl_3, V_3^\vee]$ are given by the action of $\sl_3$ on $V_3$ and $V_3^\vee$;
\item the commutators $[V_3, V_3]$ and $[V_3^\vee, V_3^\vee]$ are given by\footnote{The $2$ here is $1_J \times 1_J = 2 1_J$} $[x,y] = 2 x \wedge y \in V_3^\vee$ for $x,y \in V_3$ and $[\gamma, \delta] = 2 \gamma \wedge \delta \in V_3$ for $\gamma, \delta \in V_3^\vee$;
\item if $x \in V_3$ and $\gamma \in V_3^\vee$ then\footnote{The $3$ here is $\tr_{J}(1_J)$} $[\gamma, x] = 3x \otimes \gamma - (x,\gamma) 1_3$, which is in $\sl_3$.
\end{enumerate}
The elements $E_{ij}$, $v_j$, $\delta_k$ of $\g$ defined above are simply the standard basis vectors for $\sl_3$, $V_3$ and $V_3^\vee$ in the decomposition $\g = \sl_3 \oplus V_3 \oplus V_3^\vee$.

\subsection{The Heisenberg parabolic subgroup}  We now describe the Heisenberg parabolic subgroup $P = MN$ of $G_2$.  This is the subgroup that stabilizes the two-dimensional subspace of $V_7$ spanned by $e_3^*$ and $e_1$.  The unipotent radical $N$ is two-step, $N \supseteq N_0 \supseteq 1$, with $N/N_0$-four-dimensional abelian, and $N_0 = [N,N]$ one-dimensional.  The Lie algebra $\n$ of $N$ is spanned by $v_1, \delta_3$, and the strictly upper triangular matrices in $\sl_3$.  The element $E_{13}$ of $\sl_3$ spans $N_0$.  

We chose the Levi subgroup $M$ of $P$ by requiring it to have Lie algebra spanned by $v_2, \delta_2$, and the Cartan $\h$.  This Levi subgroup $M$ of the Heisenberg parabolic is isomorphic to $\GL_2$.  The isomorphism $M \simeq \GL_2$ is chosen so that $Ad(m) E_{13} = \det(m) E_{13}$.  The conjugation action of $M$ on $N/N_0$ is the representation $Sym^3(V_2) \otimes \det(V_2)^{-1}$ of $\GL_2$.  On the Lie algebra level, the isomorphism $M_2(F) \simeq \mathfrak{m}$ is given by
\begin{equation} \label{gl2Ident} \mm{a}{b}{c}{d} \mapsto d E_{11} + (a-d)E_{22} - aE_{33} + b v_2 - c \delta_2.\end{equation}

\subsection{Cartan involution}\label{subsec:Cartan}  Before defining a Cartan involution $\Theta$ on $\g$ and a maximal compact subgroup $K$ of $G_2(\R)$, we briefly discuss the Killing form on $\g$.  If $V$ is a quadratic space, then the Killing form on $\mathfrak{so}(V) \simeq \wedge^2 V$ is proportional to the quadratic form given by $(w\wedge x, y\wedge z) = (w,z)(x,y) - (w,y)(x,z)$.  This pairing on $\wedge^2 V_7$ induces one on $\g$ which is proportional to the Killing form on $\g$.

Denote $\iota: \Theta\rightarrow \Theta$ the involution given by $\iota\left(\mb{a}{v}{\phi}{d}\right) = \mb{d}{-\tilde{\phi}}{-\tilde{v}}{a}$.  Here $\tilde{e_i} = e_i^*$, $\tilde{e_j^*} = e_j$ and $\tilde{\cdot}$ is extended by linearity.  The involution $\iota$ has $\pm 1$ eigenspaces on $\Theta$ each with dimension $4$, and the quadratic form $(\cdot, \cdot )$ on $\Theta$ is positive-definite on the $+1$-eigenspace, and negative-definite on the $-1$ eigenspace.  

The element $1 \in \Theta$ is fixed by $\iota$, and abusing notation we denote by $\iota$ the restriction of this involution to $V_7$.  Then $\iota$ induces an involution $\theta$ on $\wedge^2 V_7$ via $\theta( w \wedge x) = \iota(w) \wedge \iota(x)$.  Since $\iota(u_0) = -u_0$, $\iota(e_i) = -e_i^*$, and $\iota(e_i^*) = -e_i$, it is clear from our basis above that $\theta$ preserves $\g \subseteq \wedge^2 V_7$.

\begin{claim} The involution $\theta$ is a Cartan involution on $\g$.\end{claim}
\begin{proof} Because $\iota \in O(V_7)$, $\theta$ preserves the bracket on $\mathfrak{so}(V_7)$, and thus also on $\g$.  It is clear as well that $\theta$ is an involution, because $\iota$ is. That $\theta$ is a Cartan involution now follows easily from the definition: I.e., $B_{\theta}(x,y) := - B(x,\theta(y))$ is positive-definite on $\g$, where $B$ is the Killing form.\end{proof}

The maximal compact subgroup of $G_2(\R)$ is $(\SU(2) \times \SU(2))/\mu_2$.  Below we will see this explicitly on the Lie algebra level.

\section{Modular forms on $G_2$}\label{sec:MFs} In this section we define modular forms on $G_2$, after \cite{ganGrossSavin}.  We then discuss their ``abstract'' Fourier expansion in terms of cubic rings, which follows from a multiplicity one result of Wallach at the archimedean place and the arithmetic invariant theory of $\GL_2(\Z)$ acting on $Sym^3(\Z^2) \otimes \det^{-1}$ (the orbits parametrize cubic rings).  We then state the two main theorems that we will prove:
\begin{enumerate}
\item An explicit, or refined Fourier expansion.  This is a purely archimedean result.  It is the $G_2$-special case of the main result of \cite{pollackQDS}.
\item A Dirichlet series for the standard $L$-function of a modular form on $G_2$, away from finitely many bad primes.  This is really a $p$-adic result (for $p$ finite).  With a small amount of work, it can be extracted from the papers \cite{gurevichSegal} and \cite{segal}.  Namely, these papers give Rankin-Selberg integrals for cusp forms on $G_2$, and the Dirichlet series is essentially equivalent to the unramified computation in these two papers.
\end{enumerate}
We hope people will be interested in this beautiful Rankin-Selberg integral of Gurevich-Segal and Segal.  However, the computations in \emph{loc cit} are rather difficult (at least for this author!).  Thus we give a new presentation of these Rankin integrals, following the same outline as \cite{gurevichSegal,segal} but with a few technical simplifications.  We explain below the manner in which we accomplish these simplifications.

\subsection{Quaternionic discrete series} The definition of modular forms on $G_2$, from \cite{ganGrossSavin}, relies on that of the quaternionic discrete series representations \cite{grossWallach2}.  These are a family of certain discrete series representations $\pi_n$ of $G_2(\R)$, for integers $n \geq 2$, singled out by Gross and Wallach \cite{grossWallach1,grossWallach2}.  In this subsection, we briefly discuss these representations, and the definition of modular forms on $G_2$.

By way of analogy, first consider the group $\Sp_{2g}$, and the Siegel modular forms.  The group $\Sp_{2g}(\R)$ has holomorphic discrete series representations, and the automorphic representations of $\Sp_{2g}(\A)$ that have these as their infinite components give rise to the Siegel modular forms.  More precisely, the maximal compact subgroup of $\Sp_{2g}(\R)$ is the unitary group $U(g)$.  Consider the holomorphic discrete series representation $\sigma_n$ of $\Sp_{2g}(\R)$, whose minimal $K$-type is the one-dimensional representation $\det^{\otimes n}$ of $U(g)$.  The representations $\sigma_n$ on $\Sp_{2n}(\R)$ have the smallest Gelfand-Kirillov dimension of the discrete series on $\Sp_{2n}$ \cite{wallachGK}, and simplest minimal $K$-type.  These representations are responsible for the Siegel modular forms on $\Sp_{2g}$ of weight $n$.

The discrete series representations $\pi_n$ on $G_2(\R)$ are analogues of the representations $\sigma_n$; they too have the smallest Gelfand-Kirollov dimension among the discrete series on $G_2(\R)$ \cite{wallachGK}, and have the simplest minimal $K$-type.  More precisely, set $K \simeq (\SU(2) \times \SU(2))/\mu_2$ the maximal compact subgroup of $G_2(\R)$, that corresponds to the Cartan involution $\theta$ defined in subsection \ref{subsec:Cartan}.  From \cite{grossWallach2}, one has the $K$-type decomposition
\begin{equation}\label{eqn:Ktypes} \pi_n|_{K} = \bigoplus_{k \geq 0} Sym^{2n+k}(\C^2) \boxtimes Sym^{k}(W_{\C}).\end{equation}
Here $\C^2$ is the defining two-dimensional representation of the ``long root'' $\SU(2)$, which is the first $\SU(2)$-factor in $K = (\SU(2) \times \SU(2))/\mu_2$, and $W_{\C}$ is the four-dimensional irreducible representation of the ``short root'' $\SU(2)$, which is the second $\SU(2)$-factor in $K$.  (The summands in (\ref{eqn:Ktypes}) are not generally irreducible.) Thus the minimal $K$-type of $\pi_n$ is $\V_n:=Sym^{2n}(\C^2) \boxtimes \mathbf{1}$.  Note that $K$ has no nontrivial one-dimensional representations, and thus there are no discrete series of $G_2(\R)$ with one-dimensional minimal $K$-types.

Roughly, the modular forms on $G_2$ of weight $n$ correspond to certain automorphic forms in automorphic representations $\pi = \pi_f \otimes \pi_{\infty}$ with $\pi_{\infty} = \pi_n$.  More precisely, one defines\footnote{In the paper \cite{pollackQDS} we generalize slightly the definition of modular forms on $G_2$; for example, one can also define modular forms of weight $1$.} modular forms on $G_2$ as follows. Denote by $\mathcal{A}(G_2)$ the space of automorphic forms on $G_2(\A)$.  Then, following \cite{ganGrossSavin}, one sets
\[\mathrm{MF}(G_2)_n = \Hom_{G_2(\R)}(\pi_n,\mathcal{A}(G_2))\]
the space of modular forms of weight $n$ on $G_2$.  More concretely, associated to a homomorphism $\varphi \in \Hom_{G_2(\R)}(\pi_n,\mathcal{A}(G_2))$, one gets a ``vector-valued'' function $F_{\varphi}$ on $G_2(\A)$ by restricting $\varphi$ to the minimal $K$-type $\V_n$ of $\pi_n$.  That is, one defines $F_{\varphi}: G_2(\A) \rightarrow \V_n^\vee$ as $F_{\varphi}(g)(v) = \varphi(v)(g)$.  We will think about $\mathrm{MF}(G_2)_n$, the modular forms on $G_2$ of weight $n$, concretely through their associated vector-valued functions $F_{\varphi}$.

\subsection{Wallach's result and Fourier coefficients} Suppose $\varphi$ is a modular form of weight $n$ on $G_2$.  We now discuss the definition of the Fourier coefficients of $\varphi$, or equivalently the Fourier coefficients of $F_{\varphi}$.  We will consider characters of $N(\Q)\backslash N(\A)$, where recall $N$ denotes the unipotent radical of the Heisenberg parabolic.  Thus suppose $\chi: N(\Q)\backslash N(\A) \rightarrow \C^\times$ denotes a unitary character of $N$. (In particular, $\chi$ is trivial on $N_0(\A)$, because $N_0 = [N,N]$ is the commutator subgroup of $N$.) If $\xi$ is an automorphic form on $G_2$, then $\xi_{\chi}(g):=\int_{N(\Q)\backslash N(\A)}{\chi^{-1}(n)\xi(ng)\,dn}$ is a function $G_2(\A)\rightarrow \C$ satisfying $\xi(ng) = \chi(n)\xi(g)$ for all $n \in N(\A)$.  Similarly, we define $F_{\varphi,\chi}: G_2(\A) \rightarrow \V_n^\vee$ as
\[F_{\varphi,\chi}(g) = \int_{N(\Q)\backslash N(\A)}{\chi^{-1}(n)F_{\varphi}(ng)\,dn}.\]

Of course, for general automorphic forms $\xi$ on $G_2(\A)$, not much can be said about the function $\xi_{\chi}$ on $G_2$.  However, when $\varphi$ is a modular form, it turns out that the function $F_{\varphi,\chi}$ restricted to the real points of the Levi subgroup $M(\R) = \GL_2(\R)$ is a particular nice function, and that much arithmetical information about $\varphi$ can be extracted from $F_{\varphi,\chi}$.  Elaborating on these claims is what this paper is about.  

The first nontrivial fact in this direction is the following multiplicity one result of Wallach, which is a special case of the main result of \cite{wallach}.  In the following theorem, a character $\chi: N(\R) \rightarrow \C^\times$ is called generic if it is in the open orbit for action of $M(\C)$ on $N^{ab}(\C)^\vee$.
\begin{theorem}[Wallach] Suppose that $\chi: N(\R) \rightarrow \C^\times$ is a generic character.  Then 
\[\dim \Hom_{N(\R)}(\pi_n,\chi) \leq 1,\]
i.e., there is at most a one-dimensional space of $\chi$-equivariant moderate growth linear functionals on $\pi_n$.\end{theorem}
 
In fact, the single open orbit of $M(\C)$ on $N^{ab}(\C)^\vee$ breaks up into two orbits of $M(\R)$ on $N^{ab}(\R)^\vee$, a ``positive'' one and a ``negative'' one.  Wallach proves that $\Hom_{N(\R)}(\pi_n,\chi) = 0$ if $\chi$ is in the negative orbit, $\Hom_{N(\R)}(\pi_n,\chi)$ is one-dimensional if $\chi$ is in the positive orbit, and furthermore determines $\Hom_{N(\R)}(\pi_n,\chi)$ as a representation of the stabilizer of $\chi$ in $M$.

From now on we denote $W = N^{ab}$, so that\footnote{The Levi subgroup $M$ of $P$ is conjugate in $G_2(\C)$ to the short root $\SU(2)$, and through this conjugation, $W \otimes \C$ becomes identified with the $W_{\C}$ appearing in the $K$-type decomposition of $\pi_n$.  This is why we use the same letter for both spaces.} $W = Sym^3(V_2) \otimes \det^{-1}$ as a representation of $M \simeq \GL_2$.  Using the above results of Wallach, Gan-Gross-Savin \cite{ganGrossSavin} defined a notion of Fourier coefficient for modular forms $\varphi$ on $G_2$, as follows.

Fix $\chi_{0}: N(\R) \rightarrow \C^\times$ a non-degenerate character in the ``positive'' orbit of $M(\R)$, and a basis $\ell_0$ of $\Hom_{N(\R)}(\pi_n,\chi_0)$.  Suppose that $\chi: N(\Q) \backslash N(\A) \rightarrow \C^\times$ is a non-degenerate character, and that $\chi|_{N(\R)}$ is in the positive orbit.  Then using $\chi_0$ and Wallach's result on the stabilizer, one can make a canonical basis element $\ell_{\chi} \in \Hom_{N(\R)}(\pi_n,\chi)$ (that depends on $\chi_0$.)

Now, suppose $\varphi$ is a modular form on $G_2$ of weight $n$.  Then integration
\[v \in \pi_n \mapsto \int_{N(\Q)\backslash N(\A)}{\chi^{-1}(n) \varphi(v)(n)\,dn}\]
defines an element $I_{\varphi,\chi}$ of $\Hom_{N(\R)}(\pi_n,\chi)$.  Thus, there is $a_{\varphi}(\chi) \in \C$ such that $I_{\varphi,\chi} = a_{\varphi}(\chi) \ell_{\chi}$ as elements of $\Hom_{N(\R)}(\pi_n,\chi)$.
\begin{definition}[Gan-Gross-Savin] Suppose $\varphi$ is a modular form on $G_2$ of weight $n$, and $\chi_0, \ell_0$ have been fixed as above.  Then the complex number $a_\varphi(\chi)$ is called the $\chi$-Fourier coefficient of $\varphi$.  If $\chi|_{N(\R)}$ is in the negative orbit, then the $\chi$-Fourier coefficient of $\varphi$ is $0$. \end{definition}

It is a consequence of the first main result that we will prove that the definition of Fourier coefficients of $\varphi$ can also be extended to the characters $\chi$ on $N(\Q) \backslash N(\A)$ that are not-necessarily generic, so long as they are not the trivial character.  Cusp forms, however, only have generic Fourier coefficients.

\subsection{Cubic rings and Fourier coefficients}\label{subsec:cubicRingsFCs} We have now briefly explained the analytic result of Wallach that goes into defining the Fourier coefficients of modular forms on $G_2$.  Using the arithmetic invariant theory of cubic rings, which parametrize the orbits of $\GL_2(\Z)$ on $W(\Z) = Sym^3(\Z^2) \otimes \det^{-1}$, Gan-Gross-Savin refined and enriched the above notion of Fourier coefficients of modular forms.

To understand this, first fix the standard additive character $\psi: \Q\backslash \A \rightarrow \C^\times$.  With $\psi$ fixed, characters $\chi$ on $N(\Q)\backslash N(\A)$ correspond to elements $w \in W(\Q)$.  Namely, $W$ has a symplectic form $\langle \;,\;\rangle$ satisfying $\langle g w_1, g w_2 \rangle =\det(g) \langle w_1, w_2 \rangle$ for all $w_1, w_2 \in W$ and $g \in \GL_2$.  Then, associated to $w \in W(\Q)$ one defines $\chi_w: N(\Q)\backslash N(\A) \rightarrow \C^\times$ as $\chi_w(n) = \psi(\langle w, \overline{n} \rangle)$, where $\overline{n}$ is the image of $n$ in $N^{ab} \simeq W$.  For $w \in W(\Q)$ generic and $\varphi$ a modular form, we set $a_{\varphi}(w) := a_{\varphi}(\chi_w)$.

Now, suppose that $\varphi$ is a modular form of weight $n$ and \emph{level one} for $G_2$.  That $\varphi$ is level one means that $F_{\varphi}(gk) = F_{\varphi}(g)$ for all $k \in G_2(\widehat{\Z})$.  It then follows easily that $a_{\varphi}(w) \neq 0$ implies $w \in W(\Z)$.  Furthermore, if $\varphi$ is of weight $n$, then $a_\varphi(k \cdot w) = \det(k)^{n} a_\varphi(w)$ for all $k \in \GL_2(\Z)$ \cite[pg 125]{ganGrossSavin}.  Consequently, if the weight $n$ of $\varphi$ is even, the Fourier coefficients of $\varphi$ are invariants of the orbits of $\GL_2(\Z)$ acting on $W(\Z)$.

As mentioned, the orbits of $\GL_2(\Z)$ on $W(\Z)$ parametrize cubic rings over $\Z$.  A cubic ring $T$ over $\Z$ is a commutative ring that is isomorphic to $\Z^3$ as a $\Z$-module.  The parametrization is as follows \cite{deloneFaddeev,ganGrossSavin,grossLucianovic}.  An element of $W(\Z)$ is binary cubic form $f(x,y) = ax^3 + bx^2y + cxy^2 + dy^3$, with $a,b,c,d \in \Z$.  Associated to $f$, one sets $T = \Z \oplus \Z \omega \oplus \Z \theta$, with multiplication table
\begin{itemize}
\item $\omega \theta = -ad$
\item $\omega^2 = -ac + a\theta - b\omega$
\item $\theta^2 = -bd + c\theta - d \omega$.
\end{itemize}
The basis $1, \omega,\theta$ of $T$ has the special property that $\omega \theta \in \Z$.  One calls bases $1,\omega,\theta$ of cubic rings $T$ with this property \emph{good bases}.  It is easily seen that every cubic ring has a good basis.  The above association $f(x,y) \mapsto (T,\{1,\omega,\theta\})$ induces a bijection between binary cubic forms and isomorphism classes of cubic rings with a good basis.  There is an action of $\GL_2(\Z)$ on the set of good bases of a fixed cubic ring $T$, and the association $f(x,y) \mapsto (T,\{1,\omega,\theta\})$ is equivariant for the action of $\GL_2(\Z)$ on both sides.  It induces a bijection between orbits of $\GL_2(\Z)$ on $W(\Z)$ and isomorphism classes of cubic rings.

One can also see \cite[Section 2.3]{pollackLL} for a discussion of the parametrization of cubic rings, from the point of view of ``lifting laws''.  The lifting law will appear when we consider the Rankin-Selberg integral in section \ref{sec:RS}.

As a consequence of the above discussion, one can parametrize the Fourier coefficients of a modular form on $G_2$ in terms of cubic rings.
\begin{definition}[Gan-Gross-Savin] Suppose that $\varphi$ is a modular form on $G_2$ of level one and even weight, and that $T$ is a non-degenerate cubic ring over $\Z$, i.e., assume $T \otimes_{\Z} \Q$ is an \'{e}tale $\Q$-algebra.  Define $a_{\varphi}(T) := a_{\varphi}(w)$, where $w \in W(\Z)$ is any element of the $\GL_2(\Z)$-orbit on $W(\Z)$ that corresponds to the ring $T$.\end{definition}
Over $\R$, there are two types of non-degenerate cubic rings: $\R \times \R \times \R$ and $\R \times \C$.  The first type corresponds to the ``positive'' orbit of $\GL_2(\R)$ on $W(\R)$, and the second to the ``negative'' orbit.  Thus, if $T$ is a cubic ring and $T \otimes \R = \R \times \C$, then $a_{\varphi}(T) = 0$.

\subsection{The theorems} We now state the two main results that we will prove below.  

\subsubsection{The Fourier expansion} The first theorem is the analytic description of the Fourier expansion of modular forms of weight $n$ on $G_2$.  To setup the result, if $F=F_{\varphi}: G_2(\A) \rightarrow \V_n^\vee$ is associated to a modular form $\varphi$ on $G_2$ of weight $n$, set $F_0(g) = \int_{N_0(\Q)\backslash N_0(\A)}{F(ng)\,dn}$.  Then we can Fourier expand $F_0$ along $W = N/N_0$.

To do this, first recall that the representation $W$ of $\GL_2$ has an invariant symplectic form $\langle \;,\; \rangle$, mentioned above.  If $w \in W$ corresponds to the binary cubic form $f(x,y) = ax^3 + bx^2y + cxy^2 + dy^3$, and $w' \in W$ corresponds to the binary cubic $f'(x,y) = a'x^3 + b'x^2y + c'xy^2 + d'y^3$, then $\langle w,w' \rangle = ad'- \frac{1}{3}bc' + \frac{1}{3}cb' - da'.$  Instead of writing out binary cubic forms, from now on we write elements of $W$ as four-tuples $w= (a,\frac{b}{3},\frac{c}{3},d)$; this element corresponds to the binary cubic $ax^3 + bx^2y + cxy^2 + dy^3$.

Associated to an element $w \in W(\R)$, we have the polynomial $h_w(z) = az^3 + bz^2 + cz +d$ on $\C$.  We write $w \geq 0$ if $h_w(z)$ has no roots on the upper (equivalently) lower half-plane $\mathcal{H}_1^{\pm}$. This notion plays an important role in the statement of Theorem \ref{thm:FE}, i.e., in the Fourier expansion of modular forms on $G_2$.

Here are some facts, examples and non-examples:
\begin{enumerate} 
\item Suppose $w = \frac{1}{3}(0,1,-1,0)$.  Then $h_w(z) = z^2-z=z(z-1)$, which has only real roots.  Thus $w \geq 0$.  
\item Suppose $w = (1,0,0,d)$.  Then $h_w(z) = x^3 + d$.  Hence if $d \neq 0$, $w$ is not $\geq 0$, but if $d = 0$, then $w \geq 0$.
\item The condition $w \geq 0$ is equivalent to the binary cubic $f_w(x,y) = ax^3 + bx^2y + cxy^2 + dy^3$ factoring into three linear factors over $\R$, and thus is an invariant of the orbit $\GL_2(\R)w$.  
\item If $w$ is in the open orbit for $\GL_2(\C)$ on $W(\C)$, so that associated to $w$ is a cubic etale $\Q$-algebra $E_w$, then $w \geq 0$ if and only if $E_w \otimes \R = \R \times \R \times \R$.
\item Note that \emph{the degree} of the polynomial $h_w$ can change in the $\GL_2(\R)$-orbit, even though the condition $w \geq 0$ is invariant.  The reason we define the condition $w \geq 0$ in terms of the single variable polynomial $h_w$ not having $0$'s, as opposed to the binary cubic $f_w$ factoring, is that the condition $h_w$ not having $0$'s on $\mathcal{H}_1^{\pm}$ generalizes nicely to the larger quaternionic exceptional groups.
\end{enumerate}

Just like the Fourier expansion of Siegel modular forms on $\Sp_{2n}(\R)$ involves the special functions $e^{2\pi i \tr(TZ)}$ on the Siegel upper half-space, the Fourier expansion of modular forms on $G_2$ involves some special functions.  In preparation for the statement of Theorem \ref{thm:FE}, we now define these functions.

Fix a positive integer $n$.  For $w \in W(\R)$ with $w \geq 0$, and an integer $v$ with $-n \leq v \leq n$, define the function $\mathcal{W}_w^v: \GL_2(\R) \rightarrow \C$ as
\begin{equation}\label{eqn:Wwv}\mathcal{W}_w^v(m) = \left(\frac{ |j(m,i) h_w(z)|}{j(m,i)h_w(z)}\right)^{v} \det(m)^{n} |\det(m)| K_v(|j(m,i)h_w(z)|).\end{equation}
Here, $K_v$ denotes the $v^{th}$ K-Bessel function.  Furthermore, for $m \in \GL_2(\R)$, $z = m \cdot i$ in $\mathcal{H}_1^{\pm}$ and $j(m,i) = (ci+d)^3 \det(m)^{-1}$ if $m = \mm{a}{b}{c}{d}$.  Note that it is important that $h_w(z)$ has no zeros on $\mathcal{H}_1$ for this function to be defined, because the Bessel function $K_v$ is not defined at $0$.

Now, fix the standard basis $x, y$ of weight vectors of the defining representation of the long-root $\SU(2)$, so that $\{x^{n+v}y^{n-v}\}_{-n \leq v \leq n}$ are a basis of $\V_n = Sym^{2n}(V_2) \boxtimes \mathbf{1}$.  Define $\mathcal{W}_w: \GL_2(\R) \rightarrow \V_n^\vee$ as
\[\mathcal{W}_w(m) = \sum_{-n \leq v \leq n}{\mathcal{W}_w^{v} \frac{x^{n+v}y^{n-v}}{(n+v)!(n-v)!}}.\]
The function $\mathcal{W}_w$ is the special function that controls the Fourier expansion of modular forms on $G_2$, as stated in Theorem \ref{thm:FE}.

Finally, denote $n: W \rightarrow N/N_0$ the identification of the $\Q$-vector space $W$ with the unipotent group $N/N_0$.  With this preparation, we have the following result, from \cite{pollackQDS}:
\begin{theorem}\label{thm:FE} Suppose $F = F_{\varphi}$ corresponds to a modular form $\varphi$ on $G_2$ of weight $n$.  Denote by $F_{00}$ the constant term of $F$ along $N$.  Then there are complex numbers $a_{F}(w)$ for $w \in 2\pi W(\Q)$, such that for all $x \in W(\R)$ and $m \in M(\R) \simeq \GL_2(\R)$ one has
\[F_0(n(x)m) = F_{00}(m) + \sum_{w \in 2\pi W(\Q), w \geq 0}{a_{F}(w) e^{-i \langle w, x \rangle}\mathcal{W}_w(m)}.\]
Furthermore, the constant term $F_{00}$ is of the form
\[F_{00}(m) = \Phi(m) x^{2n} + \beta x^n y^n + \Phi(m w_0) y^{2n}\]
with $\beta \in \C$, $w_0 = \mm{-1}{}{}{1} \in \GL_2$, and $\Phi$ associated to a holomorphic modular form of weight $3n$ on $\GL_2$.
\end{theorem}
Theorem \ref{thm:FE} gives the explicit Fourier expansion of modular forms $\varphi$ on $G_2$ of weight $n$.  This result is the $G_2$-special case of the main result of \cite{pollackQDS}.  Note that the sum over $w$ in the theorem is only over those $w$ with $w \geq 0$, just like, for example, the fact that Siegel modular forms on $\Sp_{2n}(\R)$ only have Fourier coefficients corresponding to positive semi-definite matrices $T$.

\subsubsection{The Dirichlet series} The second main result we discuss is a Dirichlet series for the standard $L$-function of cuspidal representations $\pi$ associated to modular forms on $G_2$.  This Dirichlet series follows from \cite{gurevichSegal} and \cite{segal} without too much work; it is more-or-less equivalent to the unramified computation in these two papers.

\begin{theorem}\label{thm:DS} Suppose $\pi= \pi_f \otimes \pi_k$ is a cuspidal automorphic representation of $G_2(\A)$, unramified at all finite places, and $\varphi$ is the level one modular form of weight $k$ on $G_2$ associated to $\pi$.  Suppose that $k$ is even.  Furthermore, suppose that $E$ is a cubic \'etale $\Q$-algebra with $E\otimes \R \simeq \R \times \R \times \R$, and $\mathcal{O}_E$ is the maximal order in $E$.  Finally, denote by $S$ the set of places of $\Q$ consisting of $2,3$ and all the primes that ramify in $E$. Then
\begin{equation}\label{eqn:dirSer} \sum_{T \subseteq \mathcal{O}_{E}, n \geq 1, \text{ both prime to $S$}}{\frac{a_{\varphi}(\Z + n T)}{[\mathcal{O}_E: T]^{s-k+1} n^{s}}} = a_{\varphi}(\mathcal{O}_E) \frac{L^{S}(\pi,Std,s-2k+1)}{L^{S}(E,s-2k+2)\zeta_{\Q}^{S}(2s-4k+2)}.\end{equation}
Here $L(E,s) = \zeta_{E}(s)/\zeta_{\Q}(s)$ is the ratio of the Dedekind zeta functions of $E$ and $\Q$; that $T$ and $n$ are both prime to $S$ means that $[\mathcal{O}_E:T]$ and $n$ are not divisible by primes in $S$; and the superscript $S$ on the $L$-functions on the right-hand side of (\ref{eqn:dirSer}) means that the local factor at primes in $S$ have been removed.\end{theorem}

Note that if $T$ is a cubic ring, then $\Z + nT$ is again a cubic ring, so the expression on the left-hand side of (\ref{eqn:dirSer}) makes sense.  Also, it would of course be desirable to improve this result so that the places in $S$ are accounted for. 

It is a fact (see \cite[Corollary 4.2]{pollackSMF}) that the Dirichlet series for the Spin $L$-function of Siegel modular forms on $\PGSp_6$ has a remarkably similar expression. In the case of Siegel modular forms on $\GSp_6$, the Fourier coefficients turn out to be parametrized by orders in quaternion algebras that are ramified at infinity (see, e.g., \cite{grossLucianovic}), and the Dirichlet series for the Spin $L$-function turns out to be a very similar sum as appears in the left-hand side of \eqref{eqn:dirSer}.

\section{The generalized Whittaker function}\label{sec:genWhit} In this section we sketch  the proof of Theorem \ref{thm:FE}.  This section is completely independent of section \ref{sec:RS}; thus the reader interested in the contents of that section may now skip there.  

Recall that $\varphi: \pi_n \rightarrow \mathcal{A}(G_2(\A))$ is a $G_2(\R)$-equivariant homomorphism, and 
\[F = F_{\varphi}: G_2(\Q)\backslash G_2(\A) \rightarrow \V_n^\vee\]
is defined as $F_{\varphi}(g)(v) = \varphi(v)(g)$.  Because $F$ is defined through the minimal $K$-type $\V_n$ of $\pi_n$, it is annihilated by a certain first-order linear differential operator $\mathcal{D}_n$.  We will shortly describe this operator, which is commonly referred to as the Schmid operator.

Suppose $\chi: N(\R) \rightarrow \C^\times$ is a character.  The Fourier expansion of $F$ along $N/N_0$ is controlled by functions $\W^\chi$ satisfying $\W^\chi(ng) = \chi(n)\W^\chi(g)$ for all $n \in N(\R)$ and $\W^\chi(gk) = k^{-1} \cdot \W^\chi(g)$ for all $k \in K \subseteq G_2(\R)$ the maximal compact.  In order to prove Theorem \ref{thm:FE}, we find all these functions $\W^\chi$ that are annihilated by $\mathcal{D}_n$ and have moderate growth.  It turns out that so long as $\chi$ is nontrivial, there is at most a one-dimensional space of such functions, and that the differential equations $\mathcal{D}_n \W^\chi = 0$ allow one to solve for this function explicitly.  The outcome of this calculation immediately results in Theorem \ref{thm:FE}.

Let us now describe the differential operator $\mathcal{D}_n$ more precisely.  Recall the Cartan involution $\theta$, and let $\g = \k \oplus \p$ denote the decomposition into $\theta=1$ and $\theta =-1$ parts.  Suppose $\{X_i\}_i$ is a basis of $\p$, and $\{X_i^*\}_i$ the dual basis of $\p^\vee$.  Set
\[\widetilde{D}F(g) = \sum_{i}{X_iF(g) \otimes X_i^*},\]
a function $G_2(\A) \rightarrow \V_n^\vee \otimes \p^\vee$.  It is independent of the choice of basis $X_i$.  

Now, as a representation of $\SU(2) \times \SU(2)$, $\p \simeq V_2 \boxtimes W$, where $W$ is the symmetric cube of the defining two-dimensional representation of the short-root $\SU(2)$.  Thus, $\V_n^\vee \otimes \p^\vee$ surjects onto $Sym^{2n-1}(V_2)^\vee \boxtimes W^\vee$.  The composition of $\widetilde{D}$ with this projection is the operator $\mathcal{D}_n$.  It induces a map
\[\left(C^\infty(G(\R))\otimes Sym^{2n}(V_2)^\vee \boxtimes \mathbf{1}\right)^{K} \rightarrow \left(C^\infty(G(\R))\otimes Sym^{2n-1}(V_2)^\vee \boxtimes W^\vee\right)^{K}.\]
Because the $K$-type $Sym^{2n-1}(V_2) \boxtimes W$ does not appear in $\pi_n|_K$, see (\ref{eqn:Ktypes}), one sees that $\mathcal{D}_n F_{\varphi} = 0$.  A good bit of the work to prove Theorem \ref{thm:FE} will be to make explicit the equations $\mathcal{D}_n \W^\chi = 0$.  See \cite{kosekiOda}, \cite{yamashita1,yamashita2}, and \cite{yoshinagaYamashita,yoshinagaYamashitaPJAS,yoshinaga} for some related work.

\subsection{The Cartan decomposition} In this subsection, we make some explicit computations with the Cartan decomposition $\g = \k \oplus \p$.  As mentioned, $\k \simeq \sl_2 \oplus \sl_2$, and $\p \simeq V_2 \boxtimes W$ as a representation of $\k$.  The purpose of this subsection is to record these isomorphisms explicitly.

First, observe that $\theta(v_j) = \delta_j$, and $\theta(E_{i,j}) = - E_{j,i}$.  Thus $\p = \g^{\theta = -1}$ is spanned by the Cartan $\h$, $E_{i,j} + E_{j,i}$ for $i \neq j$, and $v_j-\delta_j$.  One finds $\k = \g^{\theta = 1}$ is spanned by $v_j+\delta_j$ and $E_{i,j} - E_{j,i}$.  To make explicit the equation $\mathcal{D}_n \W^\chi =  0$, we will need bases of $\k$ and $\p$ that make clear the $\k$-action.

\subsubsection{Basis of $\mathfrak{k}$} To do this, for $j \in \{1,2,3\}$, and indices taken modulo $3$, set
\begin{align*} u_j &= \frac{1}{4}\left(E_{j+2,j+1}-E_{j+1,j+2} + v_j + \delta_j\right) \\ r_j &= \frac{1}{4}\left(3E_{j+2,j+1}-3E_{j+1,j+2} - v_j - \delta_j\right).\end{align*}
The elements $u_j$ and $r_j$ for $j = 1,2,3$ span $\k$.  In fact, the $u_j$ span the long-root $\sl_2$ and the $r_j$ span the short-root $\sl_2$.  As the reader may check by explicit computation, they satisfy the following commutation relations:

\begin{itemize}
\item $[u_i,r_j] = 0$ for all $i,j$;
\item $[u_j,u_{j+1}] = u_{j+2}$ 
\item $[r_j, r_{j+1}] = r_{j+2}$.\end{itemize}

Write $i = \sqrt{-1}$.  For later use, we set
\begin{itemize}
\item $h_u = 2i u_2$
\item $e_u = u_1 - iu_3$
\item $f_u = -u_1-iu_3$
\item $h_r = 2i r_2$
\item $e_r = r_1 - ir_3$
\item $f_r = -r_1 - i r_3$.\end{itemize}
These are $\mathfrak{sl}_2$-triples.  I.e., $[h_u,e_u] = 2 e_u$, $[h_u,f_u] = -2f_u$, $[e_u,f_u] = h_u$, and similarly for the $r$'s.  

\subsubsection{Basis of $\mathfrak{p}$} In this paragraph, we write down a good basis of $\p$.  For $j = 1,2,3$, set $y_j = v_j -\delta_j$, and $f_j = E_{j+1,j+2}+E_{j+2,j+1}$, with indices taken modulo $3$.  Together with the split cartan $\h$, these elements span $\p$.  We will now explicitly describe a basis of $\p$ that is well-suited to our $\mathfrak{sl}_2$-triples above. Here is that basis.

\begin{align*} d_3&:=f_1-y_1 + i(f_3+y_3) & h_3&:= 2(E_{33}-E_{11}) + 2i f_2  \\ d_1&:=-\frac{2}{3}(E_{11}-2E_{22}+E_{33}) + \frac{2}{3} i y_2 &  h_1&:= \frac{3f_1+y_1}{3} + i \left(\frac{3f_3-y_3}{3}\right) \\ d_{-1}&:=- \frac{3f_1+y_1}{3}+i \left(\frac{3f_3-y_3}{3}\right) & h_{-1}&:=-\frac{2}{3}(E_{11}-2E_{22}+E_{33}) - \frac{2}{3} i y_2 \\ d_{-3}&:= 2(E_{33}-E_{11}) - 2if_2 & h_{-3}&:= y_1-f_1+i(f_3+y_3). \end{align*}

Denote by $x_{\ell}, y_{\ell}$ a standard basis for the standard $2$-dimensional representation of the $\mathfrak{sl}_2$ spanned by $e_u, h_u, f_u$, and denote by $x_s, y_s$ a standard basis for the standard $2$-dimensional representation of the $\mathfrak{sl}_2$ spanned by $e_r, h_r, f_r$.  (The ``$\ell$'' and ``$s$'' are for \emph{long} and \emph{short}.)  As a representation of $\sl_2$, the triple $\{e_u,h_u,f_u\}$ spans a copy of $Sym^2(V_2)$.  We normalize the choice of $x_{\ell}, y_{\ell}$ by the correspondence $e_{u} \mapsto x_{\ell}^2$, $h_{u} \mapsto -2x_{\ell}y_{\ell}$, and $f_u \mapsto -y_{\ell}^2$.
\begin{proposition} The map $\p \rightarrow V_2 \boxtimes S^3 V_2$ sending $h_{3-2i} \mapsto x_{\ell}\boxtimes x_{s}^{3-i}y_s^{i}$, and $d_{3-2i} \mapsto y_{\ell} \boxtimes x_s^{3-i}y_{s}^{i}$ is an isomorphism of $\mathfrak{sl}_2 \oplus \mathfrak{sl}_2$ modules.  In particular, $h_u$ acts on $h_{*}$ by the eigenvalue $1$ and $d_{*}$ by the eigenvalue $-1$.  The element $h_r$ acts on $d_{k}$ and $h_{k}$ by the eigenvalue $k$. \end{proposition}
\begin{proof} Again, this is just explicit computation. \end{proof}

We will also need to know how the Killing form restricts to the $h_i$ and $d_j$'s. This is recorded in the following lemma. Recall from subsection \ref{subsec:Cartan} the pairing $(\cdot,\cdot)$ on $\g_2$, which is restricted from $\mathfrak{so}(V_7) \simeq \wedge^2 V_7$.  
\begin{lemma} One has $d_3 = - \overline{h}_{-3}$, $d_{1} = \overline{h}_{-1}$, $d_{-1} = - \overline{h}_1$, and $d_{-3} = \overline{h}_3$.  Futhermore, one has the following pairings:
\begin{itemize}
\item $(\overline{h}_{-3},h_{-3}) = 16$
\item $(\overline{h}_{-1},h_{-1}) = 16/3$
\item $(\overline{h}_1,h_{1}) = 16/3$
\item $(\overline{h}_{3},h_{3}) = 16$. \end{itemize}
Finally, $(h_i,h_j) = (\overline{h}_i,\overline{h}_j) = 0$ for all $i,j$  and $(\overline{h}_i.h_j) = 0$ if $i \neq j$.\end{lemma}

\begin{remark} In \cite{pollackQDS}, we can replace some of the above explicit computation, as follows.  As described above, $\g = \g_2$ has a $\Z/3$-grading, and one can define the Lie algebra structure through this $\Z/3$-model.  The Lie algebra $\g$ also has a $\Z/2$-grading, $\g = \g^0 \oplus \g^{1}$ and one can define the structure through this $\Z/2$-model.  All of the larger quaternionic exceptional Lie algebras also have both a $\Z/3$ and $\Z/2$ model.  Denote by $G$ the adjoint group associated to the Lie algebra $\g$.  In \cite[section 5]{pollackQDS} we write down an explicit element $\mathcal{C} \in G(\C)$ with the property that $\mathcal{C}(\g^0 \otimes \C) = \k$ and $\mathcal{C}(\g^1 \otimes \C) = \p$.  One gets good bases for $\k$ and $\p$ and all their associated structure from the structure of $\g^0$ and $\g^1$ and the fact that $\mathcal{C}$ is an automorphism of $\g \otimes \C$.

In order to define $\mathcal{C}$, we use both the $\Z/3$-model, the $\Z/2$-model, and an explicit isomorphism between them.  This element $\mathcal{C}$ is a sort of explicit, exceptional ``Cayley'' transform.  The reason for the name is that classical Cayley tranform has an analogous property: On $\mathfrak{sp}_{2n}$, the Cayley transform is an element $C \in \Sp_{2n}(\C)$ that conjugates the decomposition $\mathfrak{n} \oplus \mathfrak{m} \oplus \overline{\mathfrak{n}}$ to $\p^+ \oplus \k \oplus \p^{-}$, where $\mathfrak{m}$ is the Lie algebra of the Levi of the Siegel parabolic, $\mathfrak{n}$ is its unipotent radical, and $\overline{\mathfrak{n}}$ is the opposite unipotent.  A more abstract Cayley transform for the quaternionic Lie algebras can be found in \cite{grossWallach2}.  To keep things as elementary as possible, we have defined $\g_2$ through the trace $0$ octonions $V_7$ and used the $\Z/3$-model of $\g$, and thus not introduced the Cayley transform $\mathcal{C}$.\end{remark}

\subsection{Iwasawa decomposition} To write down the equations $\mathcal{D}_n \W^\chi = 0$, and to utilize the equivariance properties of the functions $\W^\chi$, we will make use of the Iwasawa decomposition on $\g$.  That is, if $X \in \g$, one can write $X = n + m +k$ for $n \in \mathfrak{n}$ the Lie algebra of $N(\R)$, $m \in \mathfrak{gl}_2$ the Lie algebra of the Levi $M(\R)$ of the Heisenberg parabolic $P$, and $k \in \k$.

Because the abelianization of $N$ is $Sym^3(V_2)$, we have a map $\mathfrak{n} \rightarrow Sym^3(V_2)$ sending an element $X \in \mathfrak{n}$ to a binary cubic polynomial $p_X$ in the two variables $u,v$.  Specifically, one sends
\[X = a E_{12} + \frac{b}{3} v_1 + \frac{c}{3} \delta_3 + d E_{23} + \mu E_{13} \mapsto p_{X} = au^3 + bu^2v + cuv^2 + dv^3 = (a,\frac{b}{3},\frac{c}{3},d).\]
We now record the Iwasawa decomposition of the elements $h_i$ using this polynomial map.  (Because $\W^\chi(\exp(\mu E_{13})g) = \W^\chi(g)$, the loss of the information about $E_{13}$ spanning $N_0 = [N,N]$ will not cause any problem.)

Set $\epsilon_1 = E_{22}-E_{33}$ and $\epsilon_2 = E_{11}-E_{22}$.  These are a basis of $\h$.  In fact, under the map $\mathfrak{gl}_2 \rightarrow \g_2$ from (\ref{gl2Ident}), $\epsilon_1$ is the image of $\mm{1}{0}{0}{0}$ and $\epsilon_2$ the image of $\mm{0}{0}{0}{1}$.

We have
\begin{align*} h_{3} &\equiv - 2(\epsilon_{1}+\epsilon_{2})-(h_u+h_r) \\
h_{1} &\equiv 2(v+iu)^2(v-iu)-\frac{4}{3} f_{r} \\
h_{-1} &\equiv \frac{2}{3}(\epsilon_{1}-\epsilon_{2}-2iv_{2}) + \frac{1}{3}(3h_u - h_{r}) \\
h_{-3} &\equiv -2(v+iu)^3 -4 e_{u}.\end{align*}
We are here employing the map $\mathfrak{n} \rightarrow Sym^3(V_2)$ just described.  One obtains the Iwasawa decomposition of the $\overline{h}_j$'s by taking the  complex conjugate of the above expressions.

\subsection{Actions of Lie algebra elements} To write down the equations $\mathcal{D}_n \W^\chi = 0$ concretely, we coordinatize the Levi subgroup $M \simeq \GL_2$ of the Heisenberg parabolic, and make explicit how the Lie algebra elements appearing in the Iwasawa decomposition above act in these coordinates.

First, because $\W^\chi$ has equivariance properties with respect to $N(\R)$ and the maximal compact subgroup $K$ of $G_2(\R)$, it suffices to choose coordinates for $B(\R)^{0}$, the connected component of the identity of the real points of the uppertriangular Borel subgroup of $M \simeq \GL_2$.  We use the coordinates $g= \mb{1}{x}{}{1}\mb{y^{1/2}}{}{}{y^{-1/2}} \mb{w}{}{}{w}$ for $x \in \R$, and $y, w \in \R^\times_{>0}$. 

More precisely, suppose $\W^\chi$ is a generalized Whittaker function for $\chi$ annihilated by $\mathcal{D}_n$.  That is, assume that $\W^\chi: G_2(\R) \rightarrow \V_n^\vee$ is a smooth function of moderate growth, with $\mathcal{D}_n \W^\chi = 0$, and that $\W^\chi(ngk) = \chi(n)k^{-1} \cdot \W^\chi(g)$ for all $n \in N(\R)$, $k \in K$ and $g \in G_2(\R)$.  We define $\phi: \R \times \R^\times_{>0} \times \R^\times_{>0} \rightarrow \V_n^\vee$ as
\[\phi(x,y,w) = \W^\chi\left(\mb{1}{x}{}{1}\mb{y^{1/2}}{}{}{y^{-1/2}} \mb{w}{}{}{w}\right).\]
We will translate the condition $\mathcal{D}_n \W^\chi = 0$ into explicit equations for $\phi$.

Now, because $\W^\chi(gk) = k^{-1} \cdot \W^\chi(g)$ for all $k \in K$, if $X \in \k$ then $(X\W^\chi)(g) = -X \cdot \W^\chi(g)$.  The character $\chi$ on $N(\R)$ is of the form $\chi(x) = e^{i \langle \omega , p_{x} \rangle}$ for a unique $\omega \in W$.  It follows that if $X \in \mathfrak{n}$ and $m \in M(\R)$, one has $(X\W^\chi)(m) = i\langle \omega, m \cdot p_{X}\rangle \W^\chi(m)$.  Finally, by definition, one has
\begin{align}\label{eqns:phiCoords} \nonumber (\epsilon_1 + \epsilon_2)\phi(x,y,w) &= w \partial_{w}\phi(x,y,w) \\
\nonumber (\epsilon_1 - \epsilon_2)\phi(x,y,w) &= 2y \partial_{y} \phi(x,y,w) \\
v_2\phi(x,y,w) &= y\partial_{x}\phi(x,y,w).\end{align}

\subsection{Schmid equation} In this subsection we write down the equation $\mathcal{D}_n \W^\chi = 0$ explicitly in coordinates.  First, recall that $\mathcal{D}_n = pr_{-} \circ \widetilde{D}$, where 
\[pr_{-} : \left(Sym^{2n}(V_2) \boxtimes \mathbf{1}\right) \otimes V_2 \boxtimes W \rightarrow Sym^{2n-1}(V_2) \boxtimes W\]
is the non-trivial $K$-equivariant map (unique up to scalar multiple).  We fix the map $Sym^{2n}(V_2) \otimes V_2 \rightarrow Sym^{2n-1}(V_2)$ as
\[x_{\ell}^a y_{\ell}^b \otimes v \mapsto a \langle x_{\ell},v\rangle x_{\ell}^{a-1}y_{\ell}^{b} + b\langle y_{\ell}, v \rangle x_{\ell}^{a} y_{\ell}^{b-1}\]
where $\langle \cdot,\cdot \rangle$ is the invariant symplectic form on $V_2$ with $\langle x_{\ell}, y_{\ell} \rangle = 1$.

Thus, for any function $F: G_2(\R) \rightarrow \V_n^\vee$ we have
\begin{align*} \widetilde{D} F &= -d_{3}F \otimes h_{-3} + 3d_1F \otimes h_{-1}-3d_{-1}F \otimes h_1 + d_{-3}F\otimes h_{3}\\ &\, + h_3F\otimes d_{-3} - 3h_{1}F \otimes d_{-1} + 3h_{-1}F \otimes d_{1} - h_{-3}F \otimes d_{3}.\end{align*}

Applying the Iwasawa decomposition of the elements $h_i, \overline{h_j}$ computed above, and the equivariance properties of $\W^\chi$ we obtain the following for $\widetilde{D} \W^\chi$:

\begin{align*} \widetilde{D} \W^\chi &= -2i \langle \omega, m \cdot (v-iu)^3 \rangle \W^\chi \otimes h_{-3} - 4f_u \cdot (\W^\chi \otimes h_{-3}) + 4\W^\chi \otimes d_{-3} \\
&- 2(\epsilon_2 - \epsilon_1 - 2i v_2) \W^\chi \otimes h_{-1} + (3h_u-h_r-4) \cdot (\W^\chi \otimes h_{-1}) \\
&+6i\langle \omega, m \cdot (v-iu)^2(v+iu)\rangle \W^\chi \otimes h_{1} - 4e_{r} \cdot (\W^\chi \otimes h_1) + 4\W^\chi\otimes h_{3} \\
&- 2(\epsilon_1+\epsilon_2) \W^\chi \otimes h_{3} - (h_u+h_r-4) \cdot (\W^\chi \otimes h_{3}) \\
&-2(\epsilon_1+\epsilon_2) \W^\chi \otimes d_{-3} + (h_u+h_r+4) \cdot (\W^\chi \otimes d_{-3}) \\ 
&- 6 i \langle \omega, m \cdot (v+iu)^2(v-iu)\rangle \W^\chi \otimes d_{-1} - 4f_{r} \cdot (\W^\chi \otimes d_{-1}) + 4 \W^\chi \otimes d_{-3} \\
& -2(\epsilon_2-\epsilon_1 + 2i v_2)\W^\chi \otimes d_{1} + (h_r - 3h_u - 4) \cdot (\W^\chi \otimes d_{1}) \\
&+ 2i \langle \omega, m \cdot (v+iu)^3 \rangle \W^\chi \otimes d_{3} - 4e_{u} \cdot (\W^\chi \otimes d_{3}) + 4 \W^\chi \otimes h_{3}.\end{align*}

Write $[x_{\ell}^r]$ to mean $\frac{x_{\ell}^r}{r!}$, and similarly for $y$.  For $k \in \Z$ with $-n \leq k \leq n$, define $\W^{\chi,k}: G_2(\R) \rightarrow \C$ via $\W^\chi(g) = \sum_{-n \leq k \leq n}{\W^{\chi,k} [x_{\ell}^{n+k}][y_{\ell}^{n-k}]}$.  Then applying the contraction $pr_{-}$, we get
\begin{align*} \mathcal{D}_n \W^\chi &= \sum_{k}{\mathcal{D}_{3,0}^{n,k} \W^{\chi,k} [x_{\ell}^{n+k}] [y_{\ell}^{n-k-1}] \boxtimes x_{s}^3} + \sum_{k}{\mathcal{D}_{2,1}^{n,k}\W^{\chi,k} [x_{\ell}^{n+k}] [y_{\ell}^{n-k-1}]\boxtimes x_{s}^2y_{s}} \\
& +\sum_{k}{\mathcal{D}_{1,2}^{n,k}\W^{\chi,k}[x_{\ell}^{n+k}] [y_{\ell}^{n-k-1}] \boxtimes x_{s}^1y_{s}^2} + \sum_{k}{\mathcal{D}_{0,3}^{n,k} \W^{\chi,k} [x_{\ell}^{n+k}] [y_{\ell}^{n-k-1}] \boxtimes y_{s}^3} \end{align*}
for functions $\mathcal{D}_{i,j}^{n,k} \W^{\chi,k}$ that are as follows:

\begin{align*} \frac{1}{2}\mathcal{D}_{3,0}^{n,k}\W^{\chi,k} &= (\epsilon_1+\epsilon_2 - (2n+2) -k)\W^{\chi,k} + i \langle \omega, m \cdot (u-iv)^3\rangle \W^{\chi,k+1} \\
 \frac{1}{2}\mathcal{D}_{0,3}^{n,k} \W^{\chi,k} &= -(\epsilon_1+\epsilon_2 - (2n+2) + k)\W^{\chi,k+1} + i \langle \omega, m \cdot (v-iu)^3 \rangle \W^{\chi,k} \\ 
\frac{1}{2}\mathcal{D}_{2,1}^{n,k}\W^{\chi,k} &= -3i \langle \omega, m \cdot (v-iu)^2(v+iu)\rangle \W^{\chi,k}-(\epsilon_2-\epsilon_1 + 2i v_2 + 3(k+1))\W^{\chi,k+1}\\
\frac{1}{2}\mathcal{D}_{1,2}^{n,k}\W^{\chi,k} &= -3 i \langle \omega, m \cdot (v+iu)^2(v-iu)\rangle \W^{\chi,k} + (\epsilon_2 - \epsilon_1 - 2i v_2 - 3k)\W^{\chi,k}.\end{align*} 

Since $\mathcal{D}_n\W^\chi = 0$, the above expressions yield differential-difference equations satisfied by the $\W^{\chi,k}$.  Recall that we defined $\phi(x,y,w) = \W^\chi\left(\mb{1}{x}{}{1}\mb{y^{1/2}}{}{}{y^{-1/2}} \mb{w}{}{}{w}\right)$, and similarly define $\phi_k(x,y,w): \R\times \R^\times_{>0} \times \R^\times_{>0} \rightarrow \C$ as 
\[\phi_k(x,y,w) = \W^{\chi,k}\left(\mb{1}{x}{}{1}\mb{y^{1/2}}{}{}{y^{-1/2}} \mb{w}{}{}{w}\right).\]
Applying the identities (\ref{eqns:phiCoords}), we obtain the following differential equations for the $\phi_k$.

\begin{proposition}\label{difEqs1} Write $m \in \GL_2(\R)$ for the element $\mb{1}{x}{}{1}\mb{y^{1/2}}{}{}{y^{-1/2}} \mb{w}{}{}{w}$.  The $\phi_k$ satisfy the following differential-difference equations:
\begin{enumerate}
\item $(w\partial_w - (2n+2) -k)\phi_k + i \langle \omega ,m \cdot  (v+iu)^3 \rangle \phi_{k+1} = 0$
\item $-(w\partial_{w} - (2n+2) + k)\phi_{k+1} + i \langle \omega, m \cdot (v-iu)^3 \rangle \phi_k = 0$
\item $-3i \langle \omega,  m \cdot (v-iu)^2(v+iu)\rangle \phi_{k}-(-2y\partial_{y} + 2i y\partial_x + 3(k+1))\phi_{k+1} =0$
\item $-3i \langle \omega, m \cdot (v+iu)^2(v-iu) \rangle \phi_{k+1}+ (-2y\partial_y - 2i y\partial_x - 3k)\phi_{k} = 0$.\end{enumerate}
\end{proposition}

\subsection{The formula}  In this subsection we analyze the equations of Proposition \ref{difEqs1}.  We find that if $\omega \neq 0$ then there is at most a one-dimensional space of solutions to these equations that are of moderate growth, and that the solutions are given as in (\ref{eqn:Wwv}).  In particular, we prove Thereom \ref{thm:FE}.

First, set $z = x+iy$, $\partial_z = \frac{1}{2}\left(\partial_{x}- i \partial_{y}\right)$ and $\partial_{z^*} = \frac{1}{2}(\partial_{x} +i \partial_{y})$. As is clear from Proposition \ref{difEqs1}, this variable change will help solve the above differential equations.  Furthermore, we note that if $f$ is a polynomial in $u,v$, and $m \in \GL_2$, as we have in Proposition \ref{difEqs1}, then $(m \cdot f)(u,v) = \det(m)^2 f\left(m^{-1}\left(\begin{array}{c} u \\ v \end{array}\right)\right)$.  To go further, we give expressions for the terms in Proposition \ref{difEqs1} of the form $\langle \omega, m \cdot f \rangle$ for the binary cubics $f$ appearing in this proposition.

The following lemma may be checked by direct computation.
\begin{lemma} Let the notation be as above. Then $m \cdot (u-iv)^3 = wy^{-3/2} (u-zv)^3$. Furthermore, suppose $\omega = - (a,\frac{b}{3},\frac{c}{3},d)$, and define $p_{\chi}(t) = at^3 + bt^2 + ct +d$, a cubic polynomial.  Then
\begin{enumerate}
\item $\langle \omega, m \cdot (-v+iu)^3 \rangle = -iw y^{-3/2} p_\chi(z^*)$
\item $\langle \omega, m \cdot (-v-iu)^3 \rangle = iw y^{-3/2} p_\chi(z)$
\item $\langle \omega, m \cdot 3(-v-iu)^2(-v+iu) \rangle = -w y^{-3/2}(2y p'_\chi(z) + 3i p_\chi(z)) = -2wy^{-5/2}\partial_{z}\left(p_{\chi}(z)y^{-3}\right)$
\item $\langle \omega, m \cdot 3(-v+iu)^2(-v-iu)\rangle = -2wy^{-5/2}\partial_{z^*}\left(p_{\chi}(z^{*})y^{-3}\right)$.\end{enumerate}
\end{lemma}

Applying the variable change and the lemma, we obtain the following simplified differential equations.  Define a function $G_k$ by the equality $\phi_k = w^{2n+2} G_k$.
\begin{corollary}\label{cor:Gkdifs} The $G_k$ satisfy the following differential-difference equations.
\begin{enumerate}
\item $(w\partial_w + k+1)G_{k+1} = - w y^{-3/2} p_\chi(z^*) G_k$
\item $(w \partial_w - k)G_{k} = - w y^{-3/2} p_{\chi}(z) G_{k+1}$
\item $(4iy\partial_{z} + 3k)G_k = -2i w y^{5/2}\partial_{z}(p_\chi(z) y^{-3})G_{k+1}$
\item $(4iy \partial{z^*} + 3k)G_k = -2i w y^{5/2} \partial_{z^*}(p_\chi(z^*)y^{-3})G_{k-1}$.\end{enumerate}
\end{corollary}

To solve the equations of Corollary \ref{cor:Gkdifs}, we require a couple preliminaries.  First, set $u = |p_{\chi}(z)|y^{-3/2} w$.  Note that, for a function $f$,
\begin{enumerate}
\item $(4iy \partial_{z} + 3v)(y^{-3v/2}f) =y^{-3v/2}(4iy \partial_{z}f)$
\item $(4iy\partial_{z^*} + 3v)(y^{3v/2} f) = y^{3v/2}(4iy \partial_{z*} f)$
\item $\partial_{z^*}(u) = \frac{1}{2} w y^{3/2} \frac{p_\chi(z)}{|p_\chi(z)|} \partial_{z^*}(p_\chi(z^*) y^{-3})$
\item $\partial_{z}(u) = \frac{1}{2} w y^{3/2} \frac{p_{\chi}(z^*)}{|p_\chi(z)|} \partial_{z}(p_\chi(z) y^{-3})$.
\end{enumerate}

The solution to the equations of Corollary \ref{cor:Gkdifs} will involve $K$-Bessel functions.  We recall the following standard facts regarding these functions:
\begin{theorem}The $K$-Bessel functions satisfy the following identities:
\begin{enumerate}
\item $((z\partial_{z})^2 - v^2)K_v(z) = z^2K_v(z)$
\item $-z^{-v} \partial_z(z^v K_v) = K_{v-1}$
\item $-z^{v} \partial_{z}(z^{-v} K_v) = K_{v+1}$
\item $-(z \partial_{z}-v)K_v = zK_{v+1}$
\item $-(z\partial_z+v)K_v = zK_{v-1}$.\end{enumerate}
\end{theorem}

We now come to the proof of the explicit formula.
\begin{proof}[Proof of Theorem \ref{thm:FE}] Applying the differential equations in $w$, one gets
\begin{align*}(w \partial_w + v)(w \partial_w - v)G_{v} &= (w\partial_{w}+v) \left(- w y^{-3/2} p_{\chi}(z) G_{v+1}\right) \\ &=- w y^{-3/2} p_{\chi}(z) (w\partial_{w}+v+1) \left( G_{v+1}\right) \\ &= (w y^{-3/2} p_{\chi}(z)) w y^{-3/2} p_\chi(z^*) G_v = u^2 G_{v}.\end{align*}
It follows that $G_v$ is a Bessel function in $w$.  More precisely, since $G_v$ cannot grow exponentially as $w \rightarrow \infty$, $G_v = K_v(u)Y_v(z,z^*)$ for some function $Y_v$ of $z,z^*$.

Now, using this, one gets
\[ w p_\chi(z) y^{-3/2} Y_{v+1}(z,z^*) K_{v+1}(u) = -(w \partial_{w}-v)G_{v} = |p_\chi(z)| y^{-3/2} w Y_v(z,z^*)K_{v+1}(u).\]
Hence $Y_{v+1} = \frac{|p_\chi(z)|}{p_\chi(z)} Y_v(z,z^*) = \frac{p_\chi(z)^*}{|p_\chi(z)|} Y_v(z,z^*)$.  One concludes that, on an open set where $p_{\chi}(z) \neq 0$, $Y_v(z,z^*) = \left(\frac{|p_\chi(z)|}{p_\chi(z)}\right)^{v} Y_0(z,z^*)$.

It remains to understand the function $Y_0$.  To do so, one applies the differential equations with $\partial_z$ and $\partial_{z^*}$ to show $\partial_{z}Y_0 = \partial_{z^*}Y_0 = 0$.  Thus, $Y_0$ is constant.  It follows that on an open set where $p_{\chi}(z) \neq 0$, one obtains that $\phi_k$ is proportional to $\left(\frac{|p_\chi(z)|}{p_\chi(z)}\right)^{v} w^{2n+2} K_v( |p_{\chi}(z)|y^{-3/2} w)$.  Note that if $w$ stays constant but $p_{\chi}(z) \rightarrow 0$, this function diverges.  Hence, in order for there to be any nonzero solutions to the equations $\mathcal{D}_n \W^\chi = 0$, we require that $p_{\chi}(z)$ is never $0$ on the upper half-plane $\mathcal{H}_1$.

One obtains the exact expression for the functions $\W_w(m)$ by imposing the $K$-equivariance $\W_w(g k) = k^{-1} \W_w(g)$.  Finally, one can check that the solutions found really do satisfy all the differential equations. \end{proof} 

\begin{remark} The reader wishing to analyze the constant term of modular forms on $G_2$ must simply take $\omega= 0$ in Proposition \ref{difEqs1}.  Substantial simplification occurs immediately.\end{remark}

\section{The Rankin-Selberg integral}\label{sec:RS} In this section we prove Theorem \ref{thm:DS}.  In fact, this Dirichlet series is more-or-less equivalent to the unramified computation of the Rankin-Selberg integrals in \cite{gurevichSegal,segal}, so we also recall these integrals.

\subsection{The global integral} We now recall the global Rankin-Selberg integrals of \cite{gurevichSegal} and \cite{segal}.  Fix $F$ a number field, and set $G = G_2$ the split exceptional group over $F$. The global integral depends on an auxiliary cubic \'etale extension $E$ of $F$.  Out of $E$, one can define a reductive $F$-group $G_E$ which is of type $D_4$, and for which there is a natural inclusion $G \subseteq G_E$.

\subsubsection{The group $G_E$} More precisely, first suppose $E = F \times F \times F$ is split, and $\Theta$ denotes the split octonions over $F$.  Denote by $(\cdot,\cdot,\cdot)$ the trilinear form on $\Theta$ given by $(x_1,x_2,x_3) = \tr((x_1x_2)x_3) = \tr(x_1(x_2x_3))$.  Then $G_E^{sc} = \Spin(8)$ is the group of triples of automorphisms $g = (g_1, g_2, g_3) \in \SO(\Theta)^3$ satisfying $(g_1x_1,g_2x_2,g_3x_3) = (x_1,x_2,x_3)$ for all $x_1,x_2,x_3 \in \Theta$.  The group $G = G_2$ embeds in diagonally $G_E^{sc}$ as $g \mapsto (g,g,g)$.

For a general cubic \'etale extension $E$ of $F$, one defines $G_E^{sc}$ as follows.  The $F$-vector space $E\otimes_{F} \Theta$ has a unique linear form $\ell$ that is the descent of the linear form $(\cdot, \cdot, \cdot): (F \times F \times F) \otimes_{F} \Theta \rightarrow F$ defined using the trace on $\Theta$.  The form $\ell$ satisfies $\ell(x v) = N_{E/F}(x) \ell(v)$ for all $x \in E$ and $v \in E \otimes \Theta$.  Then one defines $G_E$ to consist of the $E$-linear automorphisms $g$ of $E \otimes \Theta$ in $\SO_{E}(E \otimes \Theta)$ satisfying $\ell(g v) = \ell(v)$ for all $v \in E\otimes \Theta$.  The group $G = G_2$ embeds in $G_E^{sc}$ as the $E$-linear extension of $F$-automorphisms of $\Theta$.

Alternatively, instead of the using a simply-connected group of type $D_4$, one could use the adjoint group of type $D_4$ to define the Rankin-Selberg integral.  For this, one considers $\g_E = \sl_3 \oplus E^{\tr = 0} \oplus (V_3 \otimes E) \oplus (V_3 \otimes E)^\vee$.  One can define the structure of a Lie algebra on $\g_E$, for which $\sl_3 \oplus E^{\tr = 0}$, $(V_3 \otimes E)$, and $(V_3 \otimes E)^\vee$ are the pieces of a $\Z/3$-grading.  See \cite[section 1]{rumelhart}, or also \cite[section 4]{pollackQDS}.  One then defines $G_E^{ad}$ as the adjoint group associated to this Lie algebra.  The Lie algebra of $\g$ of $G=G_2$ embeds in $\g_E$ as $\sl_3 \oplus V_3 \oplus V_3^\vee$.  This embedding yields an action of $\g$ on $\g_E$, which gives rise to a map $G \rightarrow G_E^{ad}$.  (The group $G_2$ is simply-connected and adjoint.)  In fact, denote by $E^0$ the elements of $E$ with trace equal to $0$.  Then one has an isomorphism of $\g_2$-modules $\g_E = \g_2 \oplus V_7 \otimes E^0$ with $E^0$ a trivial $\g_2$-module.  Here the map $V_7\otimes E^0 \rightarrow \g_E$ is given as follows: for $j \in \{1,2,3\}$ and $\epsilon \in E^0$, $e_j \otimes \epsilon \mapsto -v_j \otimes \epsilon$, $e_j^* \otimes \epsilon \mapsto \delta_j \otimes \epsilon$, and $u_0 \otimes \epsilon \mapsto \Phi_{\epsilon,1}$ in the notation of \cite[section 3]{pollackLL}.

Now set $G_E = G_E^{sc}$ or $G_E^{ad}$.  The Rankin-Selberg integral is formed by defining a degenerate Eisenstein series on $G_E$ for the Heisenberg parabolic, and pulling it back to $G$.  Because we have $G \rightarrow G_E^{sc} \rightarrow G_E^{ad}$, and this Eisenstein series on $G_E^{sc}$ is the pullback of the degenerate Eisenstein series on $G_E^{ad}$, one can use either $G_E^{sc}$ or $G_E^{ad}$ to define the global integral. We have described both the simply-connected and the adjoint form of the group $G_E$, because both are convenient for calculating with this integral: The simply-connected group is convenient for unfolding the integral and doing orbit calculations, while the adjoint group is useful for doing certain calculations once the integral is unfolded.

\subsubsection{The global integral} Let us now define the global integral precisely.  Define $P_E$ to be the subgroup of $G_E$ preserving the line spanned by $E_{13} \in \sl_3 \subseteq \g_E$.  Then $P_E$ is the Heisenberg parabolic subgroup of $G_E$, and the action of $P_E$ on $E_{13}$ gives the similitude character $\nu: P_E\rightarrow \GL_1$, $Ad(p) E_{13} = \nu(p) E_{13}$.  The modulus character of $P_E$ is $\delta_{P_E}(p) = |\nu(p)|^{5}$.

The Eisenstein series on $G_E$ is defined as follows.  Suppose that $\Phi$ is a Schwartz-Bruhat function on $\g_E \otimes \A$.  Define
\[f(g,\Phi,s) = \int_{\GL_1(\A)}{\Phi(t Ad(g)^{-1} E_{13})|t|^{s}\,dt}.\]
If $p \in P_E$, then $f(pg,\Phi,s) = |\nu(p)|^{s} f(g,\Phi,s)$.  The Eisenstein series is
\[E(g,\Phi,s) = \sum_{\gamma \in P_E(F)\backslash G_E(F)}{f(\gamma g,\Phi,s)}.\]
Defining $f(g,\Phi,s)$ as an integral of a Schwartz-Bruhat function $\Phi$ on $\g_E$ is one of the modifications we make to the computation in \cite{gurevichSegal,segal}.  The use of $\Phi$ will lead to simplifications later.

Suppose that $\pi$ is a cuspidal automorphic representation on $G(\A)$, and $\varphi$ is a cusp form in the space of $\pi$.  The global integral is
\[I(\varphi,\Phi,s) = \int_{G(F)\backslash G(\A)}{\varphi(g)E(g,\Phi,s)\,dg}.\]
With our normalizations, the main results of \cite{gurevichSegal} and \cite{segal} imply that this integral represents the ratio of partial $L$-functions $\frac{L^{S}(\pi,Std,s-2)}{L^{S}(E,s-1)\zeta_{F}^{S}(2s-4)}$.

\subsection{Embedding in $\mathrm{SO}(7)$}\label{subsec:SO7} Before unfolding the global integral $I(\varphi,\Phi,s)$, we dispense with a preliminary that we will need below.  Namely, we describe explicitly how the Heisenberg parabolic subgroup $P = MN$ of $G_2$ embeds in $\SO(7)$. Of course, the $\SO(7)$ is $\SO(V_7)$, with quadratic form coming from the norm on the octonions.

Fix the ordered basis
\[e_1, e_3^*, e_2^*, u_0, -e_2, -e_1^*, -e_3\]
of $V_7$.  With this ordered basis, the Gram matrix for the bilinear form $(\cdot, \cdot)$ on $V_7$ is
\[\left(\begin{array}{ccc} & & 1_2 \\ & S & \\ 1_2 &&\end{array}\right) \text{ with } S = \left(\begin{array}{ccc} & & 1 \\ & -2 & \\ 1 \end{array}\right).\]
Denote by $V_2$ the defining representation of $\GL_2 \simeq M$.  Then, as a representation of $\GL_2$, we have $V_7 \simeq V_2 \oplus M_2^{\tr=0} \oplus V_2^\vee$. The quadratic form $S$ on $M_2^{\tr  = 0}$ is given by the determinant.  Put another way, one embeds $\GL_2 \simeq M$ into $\SO(V_7)$ as $g \mapsto \diag(g, Ad^0(g), \,^tg^{-1})$.  Here $Ad^0(g)$ is the map $M_2^{\tr=0} \rightarrow M_2^{\tr=0}$ given by $x \mapsto gxg^{-1}$.  In this setup, the identification $\mathrm{Span}(e_2^*,u_0,e_2)$ with $M_2^{\tr = 0}$ is given by $e_2^* \mapsto \mm{0}{-1}{0}{0}$, $u_0 \mapsto \mm{1}{0}{0}{-1}$, and $e_2 \mapsto \mm{0}{0}{-1}{0}$.\\

\noindent \textbf{Caution}: Our identification $M \simeq \GL_2$ is \textbf{not} the same identification that was described in sections \ref{sec:G2andOct} or \ref{sec:genWhit}.  The identification $M \simeq \GL_2$ just described differs from the previous ones by a conjugation by the matrix $\mm{}{1}{-1}{}$.  We will use this new identification of $M$ with $\GL_2$ for the rest of this section.

Now, if one associates to the element $X = u_{1} E_{12} + \frac{u_2}{3} v_1 + \frac{u_3}{3}\delta_3 + u_4 E_{23}$ the polynomial $f_{X}(x,y) = u_1 x^3 + u_2 x^2 y + u_3 xy^2 + u_4y^3$, and if $g \in \GL_2 \simeq M$, then $f_{g \cdot X}(x,y) = \det(g)^{-1}f((x,y)g)$.

\begin{claim} Write $(\alpha,\beta,\gamma,\delta)$ for the element $\alpha E_{12} +  \beta v_1 + \gamma \delta_3 + \delta E_{23}$ of $\g$.  The Heisenberg Lie algebra $\n$ is embedded as 
\[ \mu E_{13} +(\alpha,\beta,\gamma,\delta) \mapsto \left(\begin{array}{ccc} 0 & h' & -\mu J_2 \\ 0 & 0 & h \\ 0 & 0 & 0 \end{array}\right) \text{ with } h' = -\,^th S\]
and $h = \left(\begin{array}{cc} \alpha & \beta \\ \beta & \gamma \\ \gamma & \delta \end{array}\right)$.  Here $J_2 = \mm{}{1}{-1}{}$.\end{claim}
\begin{proof} Recall that $v_1 = u_0 \wedge e_1 + e_2^* \wedge e_3^*$ and $\delta_3 = u_0 \wedge e_3^* + e_1 \wedge e_2$. Thus
\[ \alpha E_{12} + \beta v_1 + \gamma \delta_3 + \delta E_{23} = \alpha e_2^* \wedge e_1 + \beta(u_0 \wedge e_1 + e_2^* \wedge e_3^*) + \gamma(u_0 \wedge e_3^* + e_1 \wedge e_2) + \delta e_3^* \wedge e_2.\]
Evaluating this Lie algebra element on $-e_1^*$ and $-e_3$ we get
\[\left(\alpha E_{12} + \beta v_1 + \gamma \delta_3 + \delta E_{23}\right)(-e_1^*) = \alpha e_2^* + \beta u_0 + \gamma(-e_2)\]
and
\[\left(\alpha E_{12} + \beta v_1 + \gamma \delta_3 + \delta E_{23}\right)(-e_3) = \beta e_2^* + \gamma u_0 + \delta(-e_2).\]
The claim follows. \end{proof}

The Heisenberg subgroup $N$ consists of elements of the form
\[\left(\begin{array}{ccc} 1_2 & h' & x \\ & 1_3 & h \\ &&1_2 \end{array}\right) \text{ with } \,^thSh + x + \,^tx = 0.\]

\subsubsection{Unfolding} For the detailed unfolding of the integral $I(\varphi,\Phi,s)$, we refer the reader to \cite[Theorem 3.2]{segal}.  We will content ourselves with just indicating the lines of the proof, because we state the unfolding a bit differently than it appears in \emph{loc cit}.

As preparation, denote by $\mathcal{O}_E$ the maximal order in $E$, and recall from subsection \ref{subsec:cubicRingsFCs} the notion of a good basis of $\mathcal{O}_{E}$.  Suppose $1,\omega,\theta$ is a good basis of $\mathcal{O}_E$, with multiplication table as in subsection \ref{subsec:cubicRingsFCs} with the constants $a,b,c,d$.  Set 
\[v_{E} = a E_{12} + \frac{b}{3} v_1 + \frac{c}{3} \delta_3 + d E_{23} \in \g = \sl_3 \oplus V_3 \oplus V_3^\vee\]
and
\[\widetilde{v_E} = a E_{12} - v_1 \otimes \omega + \delta_3 \otimes \theta  + d E_{23} \in \g_E = (\sl_3 \oplus E^{\tr=0}) \oplus (V_3 \otimes E) \oplus (V_3^\vee \otimes E).\]
The element $\widetilde{v_E}$ is the ``rank one lift'' of the element $v_E$ in the sense of \cite[Section 2.3]{pollackLL}.  

One can show that $\widetilde{v_E}$ is in the $G_E(F)$-orbit of $E_{13}$; say $\gamma_0^{-1} E_{13} = \widetilde{v_E}$.  In fact, one has
\[ \exp\left(ad\left(-aE_{32} + v_3 \otimes \omega + \delta_1 \otimes \theta + d E_{21}\right)\right)(E_{13}) = E_{13} + \widetilde{v_E}\]
and from this the existence of $\gamma_0$ follows easily.  Finally, define $N^{0,E} \subseteq N\subseteq G_2 = G$ the subgroup of the unipotent radical of the Heisenberg parabolic $N$ consisting of those $n$ with $\langle v_E, \overline{n}\rangle = 0$.  Here $\overline{n}$ is the image of $n$ in $N^{ab} \simeq W$ and $\langle \cdot, \cdot \rangle$ is the symplectic form on $W$.

The global integral unfolds as follows.  For a cusp form $\varphi$ on $G$, and a character $\chi$ on $N(\Q)\backslash N(\A)$, define 
\[\varphi_\chi(g)=\int_{N(\Q)\backslash N(\A)}{\chi^{-1}(n)\varphi(ng)\,dn}.\]
\begin{theorem}[Gurevich-Segal, Segal]\label{thm:RSunfold} Fix an additive character $\psi: F \backslash \A \rightarrow \C^\times$, and define a character $\chi = \chi_E$ on $N$ as $\chi(n) = \psi(\langle v_E, \overline{n}\rangle)$.  One has
\begin{align*} I(\varphi,\Phi,s) &= \int_{N^{0,E}(\A)\backslash G(\A)}{f(\gamma_0 g,\Phi,s)\varphi_{\chi}(g)\,dg} \\ &=  \int_{N^{0,E}(\A)\backslash \GL_1(\A) \times G(\A)}{|t|^{s}\Phi(t g^{-1} \widetilde{v_E})\varphi_\chi(g)\,dg} \\ &= \int_{N(\A) \backslash \GL_1(\A) \times G(\A)}{|t|^{s}\Phi_\chi(t,g)\varphi_\chi(g)\,dg} \end{align*}
where
\[\Phi_{\chi}(t,g) = \int_{N^{0,E}(\A)\backslash N(\A)}{\chi(n)\Phi(t g^{-1} n^{-1} \widetilde{v_E})\,dn}.\]
\end{theorem}
\begin{proof} Denote by $N_E$ the unipotent radical of the Heisenberg parabolic $P_E$.  The element $\widetilde{v_E}$ is ``rank one'' (see \cite[section 2.3]{pollackLL}) in the abelianization $N_E^{ab}$ of $N_E$.  As mentioned above, this fact is equivalent to the fact that $\widetilde{v_E}$ is in the $G_E(F)$-orbit of $E_{13}$, which is the minimal orbit of $G_E$ on its Lie algebra $\g_E$.  The line $F \widetilde{v_E}$ gives rise to a double coset in $P_E(F) \backslash G_E(F) \slash G_2(F)$, represented by $\gamma_0$.  One obtains a contribution to $I(\varphi,\Phi,s)$ of the form 
\[\int_{\left(G_2(F) \cap \gamma_0^{-1}P_E(F) \gamma_0\right)\backslash G_2(\A)}{\varphi(g) f(\gamma_0 g,\Phi,s)\,dg}.\]
The group $G_2(F) \cap \gamma_0^{-1}P_E(F) \gamma_0$, which is the stabilizer of the line $F \widetilde{v_E}$ in $G_2(F)$, can be computed without much difficulty, using the explicit decomposition of $\g_E$ as a $G_2$-module described above.  One obtains the integral as in the statement of the theorem.

The other double cosets in $P_E(F)\backslash G_E(F) \slash G_2(F)$ give rise to integrals that vanish because $\varphi$ is a cusp form; see \cite{gurevichSegal} and \cite{segal} for this.\end{proof}

We will begin the unramified computation of this integral in the next subsection.  By defining the Eisenstein section $f(g,\Phi,s)$ through the Schwartz-Bruhat function $\Phi$, we were able to write the unfolding integral in terms of $\Phi$ and the element $\widetilde{v_E}$.  This small tweak on \cite{gurevichSegal,segal} will enable us to simplify some the computations of \emph{loc cit}.

\begin{remark} The Dirichlet series of Theorem \ref{thm:DS}, the arithmetic invariant theory of cubic rings, the Rankin-Selberg integral of Theorem \ref{thm:RSunfold}, and the lifting law of \cite[section 2.3]{pollackLL} are all tied together.  Namely, out of the cubic ring $\mathcal{O}_E$, one constructs the group $G_E$ and the Rankin-Selberg integral.  This Rankin integral yields the standard $L$-function on $G_2$, whose Dirichlet series, given in Theorem \ref{thm:DS}, is most neatly expressed in terms of the arithmetic invariant theory of cubic rings and binary cubic forms.  Finally, the lifting law $v_E \rightsquigarrow \widetilde{v_E}$ controls the Rankin integral, relating the Eisenstein series on $G_E$ to the cusp form on $G$.  It also controls the arithmetic invariant theory of cubic rings; see \cite{pollackLL}.\end{remark}

\subsection{The unramified computation: Overview} In this subsection we give an overview of the proof that the integral of Theorem \ref{thm:RSunfold} represents the partial standard $L$-function $L^S(\pi,Std,s)$.  The technique to evaluate $I(\varphi,\Phi,s)$, following \cite{gurevichSegal,segal}, is that of ``non-unique models'', due to Piatetski-Shapiro and Rallis \cite{psRallisNewWay}.  Ultimately, it boils down to expressing the Fourier coefficients of automorphic forms in terms of Hecke eigenvalues.

For the rest of this section, everything is local at a finite prime unless explicitly mentioned to the contrary.  Thus $F$ denotes a $p$-adic local field, $\mathcal{O}$ its ring of integers with uniformizer $p$, and $E$ is a cubic \'etale extension of $F$. Furthermore, we assume that $p$ is prime to $2,3$ and that $\mathcal{O}_E$ is an unramified extension of $\mathcal{O}$. Here is the statement of the unramified computation.  Define the lattice
\[\g_E(\mathcal{O}_E)= (\sl_3(\mathcal{O}_E) \oplus \mathcal{O}_E^{\tr=0}) \oplus (V_3(\Z) \otimes \mathcal{O}_E) \oplus (V_3(\Z)^\vee \otimes \mathcal{O}_E).\]
\begin{theorem}\label{thm:unram} Suppose the representation $\pi_p$ is unramified, and denote by $V_p$ the space of $\pi_p$.  Suppose that $v_0$ is a spherical vector in $V_p$, and $L: V_{p} \rightarrow \C$ is a linear functional satisfying $L(  n v) = \chi(n) L(v)$ for all $n\in N(F)$ and $v \in V_{p}$.  Furthermore, suppose that $\Phi$ is the characteristic function of $\g_E(\mathcal{O}_E)$. Then
\begin{equation}\label{eqn:locUR}\int_{N(F)\backslash \GL_1(F) \times G(F)}{|t|^{s}\Phi_{\chi}(t,g) L( g v_0)\,dg} = L(v_0) \frac{L(\pi,Std,s-2)}{L(E,s-1)\zeta_{F}(2s-4)}.\end{equation}
\end{theorem}

The reader can no-doubt easily convince themselves that combining Theorem \ref{thm:unram} with Theorem \ref{thm:RSunfold} shows that the global integral $I(\varphi,\Phi,s)$ represents the ratio of global partial $L$-functions $\frac{L^{S}(\pi,Std,s-2)}{L^S(E,s-1)\zeta_F^S(2s-4)}$.  The formal argument that one makes to conclude this from Theorem \ref{thm:unram} is spelled out in many places, e.g. \cite{bfgNU,gurevichSegal,pollackShahKS}.

The first step in applying the technique of ``non-unique models'' is to write down what can be called an approximate basic function. A basic function is, roughly speaking, the precise way one encodes a local $L$-function into Hecke operators, and an approximate basic function is a Hecke operator that encodes an approximation to a local $L$-function.  In practice, making computations with basic functions is very difficult, but approximate basic functions can be much easier to manipulate, which is why they are useful.

In \cite{gurevichSegal}, the authors construct an approximate basic function on $G_2$ by pulling back to $G_2$ (a well-known) approximate basic function $\Delta_{\SO(7)}$ on $\SO(V_7)$.  We make a small modification to this construction, as follows: We define an approximate basic function for $\GL_1 \times G_2$, as opposed to just $G_2$.  Namely, for $t \in \GL_1$ and $h \in G_2$, define $\Delta(t,h) = \charf(t \cdot r_7(h) \in End(V_7(\mathcal{O})))$, where $r_7$ denotes the representation of $G_2$ on the trace $0$ elements $V_7$ of $\Theta$.  That is, $\Delta(t,h) = 1$ if the endomorphism $t \cdot r_7(h)$ preserves the integer lattice $V_7(\mathcal{O})$ and is $0$ otherwise.  The approximate basis function we will use going forward is the function $\Delta(t,h)|t|^{s}$.  

Because $\GL_1 \times G_2$ has a natural character $\GL_1 \times G_2 \rightarrow \GL_1$, this modification aligns with the philosophy of \cite{bk2, ngo} and the theme of the paper \cite{pollackGJ}. That $\Delta(t,h)|t|^{s}$ is an approximate basic function is made precise in the following proposition.  

Denote by $q$ the order of the residue field $\mathcal{O}/p$.  Define the Hecke operator $T$ on $G_2$ to be $q^{-3}$ times the characteristic function of those $g \in G_2(F)$ with $p \cdot r_7(g) \in M_7(\mathcal{O}) = End(V_7(\mathcal{O}))$.  For ease of notation, set $z = q^{-s}$. 
\begin{proposition}\label{prop:approxBasic} Recall that $v_0 \in V_p$ is a spherical vector for $\pi_p$, and suppose that $\ell: V_p \rightarrow \C$ is a linear funtional.  Then
\begin{align*}\int_{\GL_1(F) \times G(F)}{|t|^{s+3}\Delta(t,g) \ell( g v_0)\,dg\,dt} &=  (1-z)(1-q^{-1}z) L(\pi_p,Std,s) \ell(N_0 * v_0) \\ &= (1-z)(1-q^{-1}z) N_0(\pi_p,s) L(\pi_p,Std,s) \ell(v_0).\end{align*}
Here
\[N_0 = 1 + (q^{-1} +1)z + \frac{z^2}{q} + (q^{-2}+q^{-1})z^{3} + \frac{z^4}{q^2} -\frac{z^2}{q} T\]
and $N_0(\pi_p,s)$ is the meromorphic function of $s$ defined by the equality $N_0 * v_0 = N_0(\pi_p,s) v_0$.
\end{proposition}
\begin{proof} This follows from \cite[Proposition 7.1]{gurevichSegal} by integrating over the $t\in \GL_1$.\end{proof}

The proposition says that applying the Hecke operator $\Delta(t,h)|t|^{s+3}$ to the function $\ell(h v_0)$ yields the standard $L$-function of $\pi_p$, up to a correction factor that is a polynomial in $q^{-s}$.  We will apply Proposition \ref{prop:approxBasic} to the case $\ell = L$ where $L$ is as in Theorem \ref{thm:unram}.

It follows from Proposition \ref{prop:approxBasic} that in order to prove Theorem \ref{thm:unram}, it suffices to prove that 
\begin{equation}\label{eqn:toProve}\int_{\GL_1(F) \times G(F)}{|t|^{s+2}\Delta(t,g) L(g v_0)\,dt\,dg} = M(\pi_p,s) \int_{N(F)\backslash \GL_1(F)\times G(F)}{|t|^{s+1}\Phi_{\chi}(t,g)L(g v_0) \,dt\,dg}\end{equation}
where
\[M(\pi_p,s) = (1-q z)(1- z) N_0(\pi_p,s-1) L(E,s)\zeta_F(2s-2).\]
One does this by explicitly manipulating the two-sides of this expression.  This is the technique of non-unique models, in this particular case.  The rest of the paper is concerned with doing this manipulation.

\subsection{Local unramifed integral} To begin to analyze the left-hand side of (\ref{eqn:locUR}), one needs to compute $\Phi_{\chi}$ explicitly.  Denote by $m: \GL_2 \rightarrow G_2$ our chosen identification of the Levi subgroup of the Heisenberg parabolic with $\GL_2$.  Furthermore, for $h \in \GL_2$, set $\widetilde{h} = \det(h)h^{-1}$.
\begin{lemma} Let the assumptions and notation be as in Theorem \ref{thm:unram}.  For $t \in \GL_1$ and $h \in \GL_2$, define $\lambda = \lambda(t,h) = \det(h)/t$.  Then
\[\Phi_{\chi}(t,m(h)) = |\lambda| \charf(\lambda \in \mathcal{O}, \lambda^{-1} h \in M_2(\mathcal{O}), \lambda^{-1} \widetilde{h} \widetilde{v_E} \in \g_E(\mathcal{O}_E)).\]
\end{lemma}

Here and below, by $\charf(P)$ for a property $P$ we mean the function that is $1$ if $P$ is satisfied and $0$ otherwise.

\begin{proof} Suppose $n \in N$, and denote by $\overline{n}$ the image of $n$ in $N^{ab} \simeq W$.  Then
\begin{equation}\label{eqn:hVeTilde} t h^{-1} n^{-1}\widetilde{v_E} = t h^{-1} \widetilde{v_E} + t \det(h)^{-1} \langle v_E, \overline{n} \rangle E_{13}. \end{equation}
The lemma now follows easily from (\ref{eqn:hVeTilde}).\end{proof}

For $\lambda \in \GL_1$ and $h \in \GL_2$, define 
\[A_0(\lambda, h) = \charf(\lambda \in \mathcal{O}, \lambda^{-1} h \in M_2(\mathcal{O}), \lambda^{-1} \widetilde{h} \widetilde{v_E} \in \g_E(\mathcal{O}_E)).\]
This function $A_0$ is related to cubic rings over $\mathcal{O}$ that are subrings of $\mathcal{O}_E$.  To make this precise, we setup some notation and review elements of the arithmetic invariant theory of binary cubic forms and cubic rings.

First, for a cubic ring $T$ over $\mathcal{O}$, we denote by $c(T)$ the $p$-adic content of $T$.  This is the largest integer $c$ so that $T= \mathcal{O} + p^{c} T_0$ for a cubic ring $T_0$ over $\mathcal{O}$. If $T$ is a free rank three $\mathcal{O}$-module in $E$ of the form $T = \mathcal{O} \oplus \mathcal{O} \omega' \oplus \mathcal{O}\theta'$, then $c(T)$ is also defined, but could be negative.  We will use this more general notion moving forward.

Next, for an element $x = \mm{\alpha}{\beta}{\gamma}{\delta}$ of $\GL_2(F)$, denote by $T(x)$ the $\mathcal{O}$-module spanned by $1,\delta \omega - \beta \theta$ and $-\gamma \omega + \alpha \theta$.  Here $1, \omega,\theta$ is the good basis of the maximal order $\mathcal{O}_E$ fixed above.  Then $T(x)$ is a ring, i.e., $T(x)$ is closed under multiplication, if and only if $\widetilde{x} \widetilde{v_E} \in \g_E(\mathcal{O}_E)$.  This is closely connected to \cite[section 2.3]{pollackLL}.  As a consequence of this fact, one has
\[A_0(\lambda,h) = \charf(\lambda \in \mathcal{O}, T(\lambda^{-1}h) \text{ a ring}) = \charf(\lambda \in \mathcal{O}, \lambda| p^{c(T(h))}).\]

For ease of notation, we now set $c(x) := c(T(x))$ for $x \in \GL_2(F)$.  Denote by $I(s)$ the left-hand side of (\ref{eqn:locUR}), and recall that $z = q^{-s}$.  The local unramifed integral to be computed is as follows.  Applying the Iwasawa decomposition, one obtains
\begin{align*} I(s+1) &= \int_{\GL_1 \times \GL_2}{\delta_P^{-1}(m(h)) |\lambda| |\det(h)/\lambda|^{s+1} A_0(\lambda,h) L( m(h) v)\,d\lambda\,dh} \\ &= \int_{\GL_2}{|\det(h)|^{s-2} L(m(h) v)\left(\int_{\GL_1}{|\lambda|^{-s} A_0(\lambda, h)\,d\lambda}\right)\,dh}.\end{align*}

We convert this integral into a sum as follows.  The integral over $\lambda$ becomes
\[\int_{\GL_1}{|\lambda|^{-s} A_0(\lambda, h)\,d\lambda} = \sum_{0 \leq j \leq c(h)}{z^{-j}} = z^{-c(h)} \left(\frac{1 - z^{c(h)+1}}{1-z}\right).\]
For an element $h\in \GL_2(F)$, write $[h]$ for the coset $h \GL_2(\mathcal{O})$.  Whether or not $T(x)$ is closed under multiplication is idependent of the element $x \in h \GL_2(\mathcal{O})$.  Thus
\begin{align}\label{eqn:Isum1} \nonumber I(s+1) &= \frac{1}{1-z}\left(\sum_{[h] \text{ cubic ring}}{L(m(h) v) |\det(h)|^{-2} z^{\val(\det(h))-c(h)}(1-z^{c(h)+1})}\right) \\ &= \frac{1}{1-z}\left(\sum_{[h] \text{ cubic ring}}{L(m(h) v) |\det(h)|^{-2} P_h(z)}\right)\end{align}
with
\[P_h(z) = z^{\val(\det(h))-c(h)}(1-z^{c(h)+1}),\]
a polynomial in $z$.  The sum is over all cubic subrings of $\mathcal{O}_E$.  In order to prove the equality in (\ref{eqn:toProve}), we will further manipulate the expression (\ref{eqn:Isum1}) and the left-hand side of (\ref{eqn:toProve}).

\subsection{The Fourier coefficient of the approximate basic function} In this subsection, we manipulate the left-hand side of (\ref{eqn:toProve}).  More precisely, for $t \in \GL_1(F)$ and $h \in \GL_2(F)$ set
\[D_{\chi}(t,h) = \int_{N(F)}{\chi(n) \Delta(t,n m(h)) \,dn}.\]
Then the left-hand side of (\ref{eqn:toProve}) is
\[D(s):=\int_{\GL_1 \times \GL_2}{\delta_P^{-1}(m(h)) |t|^{s+2} D_{\chi}(t,h) L(m(h) v)\,dh\,dt}.\]

One of the main computations is the expression for $D_\chi(t,h)$. Set $f_{max}(x,y) = ax^3 + bx^2y + cxy^2 + dy^3$, the binary cubic form corresponding to the maximal order $\mathcal{O}_E$, together with its good basis $1, \omega, \theta$. For a binary cubic form $V$ with coefficients in $\mathcal{O}$, define $N(V)$ to be the number of $0$'s of $V$ in $\mathbb{P}^1(\mathbb{F}_q)$.  Also, for an element $h \in \GL_2(F)$, define $\val(h) \in \Z$ to be the largest integer $n$ so that $p^{-n} h \in M_2(\mathcal{O})$.
\begin{proposition} Define $x_0(h)$ by $h = p^{\val(h)}x_0(h)$, and as before $\lambda = \det(h)/t$.  Write $D'_\chi(\lambda,h) = D_\chi(t,h)$, i.e., $D_\chi'$ is the same function as $D_\chi$, except expressed in terms of the new variables $\lambda, h$.  Finally, define
\[\epsilon(x_0(h)) = \begin{cases} 1 &\mbox{if } x_0(h) \in \GL_2(\mathcal{O}) \\ 2 &\mbox{if } x_0(h) \notin \GL_2(\mathcal{O}).\end{cases}\]
Then 
\begin{align*} D'_{\chi}(\lambda,h) &= |\det(\lambda^{-1} h)|^{-1} \charf(h \in M_2(\mathcal{O}), \val(\lambda^{-1} h) \in \{0,1\}, T(x_0(h)) \text{ a ring}) \\ & \;\;\; \times \begin{cases} 1 &\mbox{if } \val(\lambda^{-1} h) = 0 \\ N(f_{max}) - \epsilon(x_0(h)) &\mbox{if } \val(\lambda^{-1} h) = 1.\end{cases}\end{align*}
\end{proposition}
\begin{proof} We sketch the proof.  From subsection \ref{subsec:SO7}, we have that $r_7(n m(h))$ is of the form
\[\left(\begin{array}{ccc} h & y' Ad^0(h) & z \,^{t}h^{-1} \\ & Ad^0(h) & y \,^{t}h^{-1} \\ && \,^{t}h^{-1}\end{array}\right)\]
with $y$ of the form $\left(\begin{array}{cc} \alpha & \beta \\ \beta & \gamma \\ \gamma & \delta \end{array}\right)$, $y' = -\,^ty S$ with $S$ given in subsection \ref{subsec:SO7}, and $z$ satisfying $z + z^{t} + \,^tySy=0$.

Because the character $\chi$ corresponds to the maximal order $\mathcal{O}_E$ of $E$, one can show that in order for $D_{\chi}(t,h)$ to be nonzero, one needs $h \in M_2(\mathcal{O})$.  Thus assume for the rest of the proof that $h \in M_2(\mathcal{O})$.  Set $x = t \,^th^{-1}$.  Thus, $\Delta(t,n m(h)) \neq 0$ if and only if
\begin{enumerate}
\item $x \in M_2(\mathcal{O})$;
\item $y x \in M_{3,2}(\mathcal{O})$;
\item $z x \in M_2(\mathcal{O})$.\end{enumerate}

The element $z$ above is of the form $\mu \mm{}{1}{-1}{} - \frac{1}{2} \,^tySy$ for some $\mu \in F$.  Note that for $y$ as above,
\[\mu \mb{}{1}{-1}{} - \frac{1}{2} \,^tySy = \mb{\beta^2-\alpha \gamma}{\frac{1}{2}(\beta \gamma - \alpha \delta) + \mu}{\frac{1}{2}(\beta \gamma - \alpha \delta) - \mu}{\gamma^2 - \beta \delta}.\]

Denote by $W'$ the set of $3 \times 2$ matrices of the form $\left(\begin{array}{cc} a' & b' \\ b' & c' \\ c' & d'\end{array}\right)$.  Now, for $x \in \GL_2(F) \cap M_2(\mathcal{O})$, define
\[V(x)=\left\{y \in W': yx \in M_{3,2}(\mathcal{O}) \text{ and } \exists \mu \in F \text{ s.t. } \left(\mu \mm{}{1}{-1}{} - \frac{1}{2} \,^tySy\right)x \in M_2(\mathcal{O})\right\}.\]
We thus have that
\begin{equation}\label{eqn:DchiNewInt} D_{\chi}(t,h) = ||x||^{-1} \int_{W'}{\chi(y)\charf(y \in V(x))\,dy}\end{equation}
where $x = t \,^th^{-1}$ and $||x||= |p|^{\val(x)}$.  

Suppose that $x = k_1 \mm{p^j}{}{}{p^k} k_2$ with $j \geq k$ and $k_1, k_2 \in \GL_2(\mathcal{O})$.  Note that $V(x) = V(x k_2^{-1})$.  Furthermore, by changing variables in the integral in (\ref{eqn:DchiNewInt}) to eliminate $k_1$ and by replacing $\chi$ by an equivalent character, we can assume that $x = \diag(p^j,p^k)$ with $j \geq k \geq 0$.  We therefore obtain
\begin{align*} D_{\chi}(t,h) &= q^{k} \int_{w=(\alpha,\beta,\gamma,\delta) \in W(F)} \psi(\langle \omega, w \rangle) \charf(p^j \alpha, p^k(\beta,\gamma,\delta)) \\& \;\;\; \times \charf(p^{j}(\beta^2 - \alpha \gamma) \in \mathcal{O}, p^{j}(\alpha \delta - \beta \gamma) \in \mathcal{O}, p^{k}(\gamma^2 - \beta \delta) \in \mathcal{O}))\,dw.\end{align*}
Here $\omega =(a,\frac{b}{3},\frac{c}{3},d)\in W(\mathcal{O})$ is a binary cubic form representing the maximal order $\mathcal{O}_E$ of $E$.  Set $r = j -k$, so that $x=p^k\diag(p^r,1)$.  For any $\alpha_0 \in \mathcal{O}$, making a change of variable $\alpha \mapsto \alpha + p^{-r}\alpha_0$ on the one hand fixes this integral, but on the other hand multiplies it by $\psi(-p^{-r} d \alpha_0)$.  It follows that for $D_{\chi}(t,h)$ to be nonzero, we must have $p^{r}|d$.  This is equivalent to $T(x_0(h))$ is a ring, so we have checked this condition of the statement of the proposition.

From now assume that $p^{r}|d$.  Making a change of variables, we obtain
\begin{equation}\label{eqn:DchiSum1} D_{\chi}(t,h) = q^{j} \sum_{w=(\alpha,\beta,\gamma,\delta) \in W(\mathcal{O})/p^k, **}{\psi\left(\frac{\langle \omega, (p^{-r}\alpha,\beta,\gamma,\delta) \rangle}{p^k}\right)} \end{equation}
where the conditions $**$ mean that
\[ \alpha \gamma - p^r \beta^2 \in p^k \mathcal{O}, \alpha \delta - p^r \beta \gamma \in p^k \mathcal{O}, \gamma^2 - \beta \delta \in p^k \mathcal{O}.\]

In case $k =0$ or $k =1$, this sum can be directly evaluated by hand, and one obtains values as in the statement of the proposition.  One is left to consider the case $k \geq 2$, i.e., the when $p^2|x$.  In the following lemma, we show that $\int_{V(x)}{\chi(y)\,dy}$ vanishes if $p^2 |x$, completing the proof of the proposition.\end{proof}

\begin{lemma} The sum on the right-hand side of (\ref{eqn:DchiSum1}) vanishes if $k \geq 2$.\end{lemma}
\begin{proof} The sum breaks up into two pieces, consisting of those $w$ with $p|w$ and those $w$ with $p \nmid w$.  Consider first the subsum consisting of those $w$ with $p|w$:
\[D_{\chi}^{p} := \sum_{w=(\alpha,\beta,\gamma,\delta) \in W(\mathcal{O})/p^k, p|w \text{ and } **}{\psi\left(\frac{\langle \omega, (p^{-r}\alpha,\beta,\gamma,\delta) \rangle}{p^k}\right)}.\]
If $w_0 = (0,b_0,c_0,d_0) \in W(\mathcal{O})$ is fixed, then making the variable change $w \mapsto w + p^{k-1}w_0$ preserves $D_{\chi}^{p}$.  However, this changes the sum by $\psi(\langle \omega, w_0 \rangle/p)$.  Because $\omega$ is non-degenerate modulo $p$, there is $w_0$ with $\psi(\langle \omega, w_0 \rangle/p) \neq 1$.  Hence $D_{\chi}^{p} = 0$, and it suffices to evaluate
\[D_{\chi}^{1} := \sum_{w=(\alpha,\beta,\gamma,\delta) \in W(\mathcal{O})/p^k, p \nmid w \text{ and } **}{\psi\left(\frac{\langle \omega, (p^{-r}\alpha,\beta,\gamma,\delta) \rangle}{p^k}\right)}.\]
In case $r=0$, the evaluation of the sum $D_{\chi}^{1}$ now becomes a fun exercise, which we leave to the reader; it can be analyzed by the same method that we will use to handle the case $r \geq 1$.  Thus assume $r \geq 1$.  In this case, if $w =(\alpha,\beta,\gamma,\delta)$ satisfies $**$, one finds that there are three disjoint cases:
\begin{enumerate}
\item $\alpha \in (\mathcal{O}/p^k)^\times$.  In this case, $w \equiv \alpha (1,\beta, p^r \beta^2, p^{2r} \beta^3)$ for $\beta \in \mathcal{O}/p^{k}$.
\item $\delta \in (\mathcal{O}/p^k)^\times$.  In this case, $w \equiv \delta (p^{r}\gamma^3,\gamma^2,\gamma,1)$ for $\gamma \in \mathcal{O}/p^{k}$.
\item Neither $\alpha$ nor $\delta$ are units.  In this case, $w \equiv \beta(\alpha,1,\gamma,\gamma^2)$ with $\alpha \gamma \equiv p^r \pmod{p^k}$ and both $\alpha$ and $\gamma$ divisible by $p$.\end{enumerate}

Consider first the case where $\delta \in (\mathcal{O}/p^k)^\times$.  Define the polynomial $g(x) = a - bx + cx^2 - dx^3$.  Then the subsum of $D_\chi^{1}$ corresponding to $\delta \in (\mathcal{O}/p^k)^\times$ becomes 
\[\sum_{\delta \in (\mathcal{O}/p^k)^\times, \gamma \in \mathcal{O}/p^k}{\psi\left(\frac{\delta g(\gamma)}{p^k}\right)}.\]
If $\gamma$ is such that $g(\gamma)$ is a unit, then the sum over $\delta$ with $\gamma$ fixed is $0$, because $k \geq 2$.  Thus one need only sum over those $\gamma$ with $g(\gamma) \equiv 0 \pmod{p}$.  Now, fix a root $\gamma_0$ of $g$ modulo $p$, and consider the sum
\[\sum_{\delta \in (\mathcal{O}/p^k)^\times, \gamma \equiv \gamma_0 \pmod{p}}{\psi\left(\frac{\delta g(\gamma)}{p^k}\right)}.\]
The domain of this sum is preserved under $\gamma \mapsto \gamma + p^{k-1}\gamma'$ for some fixed $\gamma'$.  But 
\[\frac{g(\gamma+p^{k-1}\gamma')}{p^k} \equiv \frac{g(\gamma)}{p^k} + \frac{g'(\gamma_0)}{p}\gamma'.\]
Thus the sum over $\gamma \equiv \gamma_0$ vanishes unless $g'(\gamma_0) \equiv 0 \pmod{p}$.  But then $g(x)$ modulo $p$ has a double root at $\gamma_0$, which implies that the binary cubic $f(x,y) \equiv \mu \ell^2 x \pmod{p}$ for some $\mu \in (\mathcal{O}/p)^\times$ and line $\ell$.  But this contradicts the assumption that $f$ is non-degenerate modulo $p$, so there are no such $\gamma_0$ with $g(\gamma_0) \equiv g'(\gamma_0) \equiv 0 \pmod{p}$.  It follows that the subsum consisting of those $w$ with $\delta$ a unit vanishes.

Next consider the case where $\alpha \in (\mathcal{O}/p^{k})^\times$.  Set
\[\omega' = (p^{2r}a,\frac{p^r b}{3},\frac{c}{3},p^{-r}d) = (a',\frac{b'}{3},\frac{c'}{3},d')\]
and define the polynomial $h(x) = a'x^3 - b'x^2 + c'x -d'$.  Note that $\omega'$ is integral.  Furthermore, because $f$ is non-degenerate but $p|d$, we must have $p \nmid c$ and thus $c= c'$ is a unit.  Now, the subsum of $D_\chi^1$ corresponding to $\alpha \in (\mathcal{O}/p^k)^\times$ becomes
\[\sum_{\alpha \in (\mathcal{O}/p^k)^\times, \beta \in \mathcal{O}/p^k}{\psi\left(\frac{\alpha h(\beta)}{p^k}\right)}.\]
Just like in the previous case, this sum immediately reduces to the subsum consisting of those $\beta$ equivalent to some $\beta_0$ modulo $p$ with $h(\beta_0) \equiv h'(\beta_0) \equiv 0 \pmod{p}$.  But since $p$ divides $a'$ and $b'$, but $c'$ is prime to $p$, there are no double roots $\beta_0$.  Hence this subsum of $D_{\chi}^{1}$ vanishes as well.

We are left with the final case, in which neither $\alpha$ nor $\delta$ is unit.  Set $d' = p^{-r}d$, which is in $\mathcal{O}$. In this case, we must evaluate
\[\sum_{\beta \in (\mathcal{O}/p^k)^\times, \alpha, \gamma \in \mathcal{O}/p^k ***}{\psi\left( \frac{\beta(a \gamma^2 - b \gamma + c - d' \alpha)}{p^k}\right)}\]
where the conditions $***$ are that $\alpha \gamma \equiv p^{r}\pmod{p^k}$ and both $\alpha$ and $\gamma$ are divisible by $p$.  But note that, as mentioned above, $c$ is a unit.  Thus the fact that $\alpha \equiv \gamma \equiv 0 \pmod{p}$ imply that $a \gamma^2 - b \gamma + c - d' \alpha$ is a unit mod $p$.  Hence the sum over $\beta$, with $\alpha$ and $\gamma$ fixed vanishes.  Thus this subsum of $D_\chi^{1}$ vanishes, and the proof of the lemma is complete. \end{proof}

We now reconsider the condition ``$T(x_0(h))$ a ring'' from the above proposition.  We claim that this condition is equivalent to $\val(h) = c(h)$.  Clearly, $T(x_0(h))$ a ring means that $c(h) \geq \val(h)$.  Recall $f_{max}(x,y) = ax^3 + bx^2y + cxy^2 + dy^3$, the binary cubic form corresponding to the maximal order $\mathcal{O}_E$, together with its good basis $1, \omega, \theta$.  By assumption $f_{max}(x,y)$ is non-degenerate modulo $p$.  It follows from this that $c(x_0(h))$ must be $0$.  Indeed, suppose $x_0(h) = k \mm{p^r}{}{}{1}k'$ with $k, k' \in \GL_2(\mathcal{O})$, and $f'= f \cdot k = a'x^3 + b'x^2y + c'xy^2 + d'y^3$.  Then $T(x_0(h))$ is represented by the binary cubic
\[\frac{1}{p^{r}}f_{max}((x,y) \widetilde{\diag(p^r,1)} \widetilde{k}) = \frac{1}{p^r}a' x^3 + b'x^2y + p^{r} c'x y^2 + p^{2r} d' y^3.\]
Thus, if $c(x_0(h)) \geq 1$, then $p^{r+1}|a'$ and $p|b'$, so that $f'$ and then $f_{max}$ is degenerate modulo $p$.

We therefore obtain
\[D(s) = \sum_{[h] \text{ cubic ring}}{L(h v) |\det(h)|^{-2} \charf(\val(h) = c(h)) z^{\val(\det(h))-c(h)}(1 + (N(f_{max})-\epsilon(h_0))z)}.\]
This is the expression that we will have to compare to the Hecke-manipulated local integral.  Define $N'(f_{max}, h_0) = N(f_{max}) -\epsilon(h_0)+1$.  We thus have
\[D(s) = \sum_{[h] \text{ cubic ring}}{L(h v) |\det(h)|^{-2} \charf(\val(h) = c(h)) z^{\val(\det(h))-c(h)}(1 - z + N'(f_{max},h_0)z)}.\]

For later use, we set
\[B_0(z) =  1 + (q +1)z + q z^2+ (q^2+q)z^{3} + q^2 z^4\]
so that $N_0(s-1)= B_0(z) - q z^2 T$.

\subsection{The Hecke operator $S$} Define the Hecke operator $S$ on $G_2(F)$ to be the bi-$G_2(\mathcal{O})$ invariant function on $G_2(F)$ that is the characteristic function of those $g \in G_2(F)$ with $p \cdot r_7(g) \in M_7(\mathcal{O}) \setminus p M_7(\mathcal{O})$.  Thus $T = q^{-3}(S+\mathbf{1}_{G_2(\mathcal{O})})$, where $\mathbf{1}_{G_2(\mathcal{O})}$ denotes the characteristic function of $G_2(\mathcal{O})$. In order to evaluate the right-hand side of (\ref{eqn:toProve}), we will need a coset decomposition of $S$.  This was given in \cite[sections 13 and 14]{ganGrossSavin}.

We require some notation for Hecke operators on the Levi $\GL_2$ of the Heisenberg parabolic.  Define
\[T(p) = \GL_2(\mathcal{O}) \mm{p}{}{}{1} \GL_2(\mathcal{O})\]
and
\[T(p^{-1}):= \GL_2(\mathcal{O}) \mm{p^{-1}}{}{}{1} \GL_2(\mathcal{O}).\]
One has the following proposition regarding $L(m(h) S * v_0)$, which is a conglomeration of \cite[Corollary 13.3, Proposition 14.2 and Proposition 14.5]{ganGrossSavin}, or essentially equivalent to \cite[Proposition 15.6]{ganGrossSavin}.
\begin{proposition}[Gan-Gross-Savin] Suppose $h \in \GL_2(F)$ and $L(m(h) v_0) \neq 0$.  Then 
\begin{align*} L(m(h) S * v_0) &= L( m(h p^{-1}) v_0) + q L(m(h * T(p^{-1})) v_0) + q^4 L(m(h * T(p)) v_0) + q^{6} L(m(h p) v_0) \\ &\;\;\; + \left(-1 + q^2(N(f_{max} \cdot h)-1)\right)L(h v_0).\end{align*}\end{proposition}

\subsection{Hecke operator applied to local integral} We can now apply $N_0(s-1)$ to the local integral $I(s+1)$. First, note that for $g \in \GL_2(F)$, 
\[\sum_{[h]}{L(m(h g) v) |\det(h)|^{-2} P_h(z)} = \sum_{[h]}{L(m(h) v) |\det(h)|^{-2} |\det(g)|^2 P_{h g^{-1}}(z)}.\]
We thus have that
\[ q^{-2} \sum_{[h]}{L\left(m(h) (S+1)* v\right) |\det(h)|^{-2} P_h(z)} = \sum_{[h]}{L(m(h ) v) |\det(h)|^{-2} M_h(z)}\]
with
\[M_h(z) = q^2 P_{hp}(z) + q P_{h * T(p)}(z) + P_{h * T(p^{-1})}(z) + P_{h p^{-1}}(z) + (N(f_{max} \cdot h)-1) P_h(z).\]
Here, for a Hecke operator $Y$ on $\GL_2$ with coset decomposition $Y = \sum_{i}{ a_i \left[y_i \GL_2(\mathcal{O})\right]}$, $P_{h * Y}(z) = \sum_{i}{a_i P_{h y_i}(z)}$. Hence 
\[N_0(s-1)I(s+1) = \frac{1}{1-z} \left(\sum_{[h]}{L(m(h ) v) |\det(h)|^{-2} (B_0(z)P_h(z)-z^2 M_h(z))}\right),\]
and thus
\[(1-qz)(1-z)N_0(s-1)I(s+1) = (1-qz)\sum_{[h]}{L(m(h ) v) |\det(h)|^{-2} (B_0(z)P_h(z)-z^2 M_h(z))}.\]

\subsection{The cubic ring identity} Thus, to prove the equality of equation (\ref{eqn:toProve}) and thus that global integral represents $L^{S}(\pi, Std,s-2)$, we need to verify the identity
\begin{align}\label{eqn:CRident1}\zeta_F(2s-2)L(E,s) (1-qz) (B_0(z)P_h(z)-z^2 M_h(z)) &= \charf(\val(h) = c(h)) z^{\val(\det(h))-c(h)} \\ \nonumber &\;\;\; \times (1 - z + N'(f_{max},h_0)z).\end{align}
There is one immediate simplification here: $\zeta(2s-2)(1-qz) = (1+qz)^{-1}$.

To verify (\ref{eqn:CRident1}), we need to know how the content of cubic rings change under the action of the Hecke operator $T(p)$. That is, we must understand how the multiset $\{c(h y_i)\}_i$ is related to $c(h)$, where $T(p) = \sqcup y_i \GL_2(\mathcal{O})$ is the coset decomposition of $T(p)$.  

Slightly more canonically, one can reformulate this in terms of $\mathcal{O}$-lattices $\Lambda$ in $F^2$.  Namely, fix the lattice $\Lambda_0 = \mathcal{O}^2 \subseteq F^2$, with $F^2$ thought of as row vectors.  Now, for any $y \in \GL_2(F)$ and lattice $\Lambda$ in $F^2$, we get a new lattice $\Lambda y$.  Identifying $\Lambda_0$ with the rank two $\mathcal{O}$-module $\mathcal{O}_{E}/\mathcal{O}$, one can consider the rank three $\mathcal{O}$-module $T(\Lambda)$ corresponding to $\Lambda$, defined by $T(\Lambda) = T(h)$ if $\Lambda = \Lambda_0 h$.  We also write $c(\Lambda) = c(T(\Lambda))$ for the content of $T(\Lambda)$.  In this setup, the Hecke operator $T(p)$ replaces a lattice $\Lambda$ with the $q+1$-sublattices of $\Lambda$ of index $q$.  We thus need to understand the multiset $\{c(\Lambda')\}$ where $\Lambda' \subseteq \Lambda$ is an index $q$ sublattice.

Suppose $h = p^{c} h_0$, with $f_{max} \cdot h_0$ integral and not divisible by $p$, so that the content of $h$ is $c$, i.e., $c(h) = c$.  Set $f_0 = f_{max} \cdot h_0$.  Then $f_0$ is one of three types modulo $p$:
\begin{enumerate}
\item $f_0$ is irreducible modulo $p$;
\item $f_0 = \ell q$ with $\ell$ a line and $q$ irreducible modulo $p$;
\item $f_0 = \ell_1 \ell_2 \ell_3$ with $\ell_i$ distinct lines modulo $p$;
\item $f_0 = \ell_1^2 \ell_2$ with $\ell_1, \ell_2$ distinct lines modulo $p$;
\item $f_0 = \alpha \ell^3$, with $\ell$ a line modulo $p$ and $\alpha \in (\mathcal{O}/p)^\times$\end{enumerate}

One has the following lemma.
\begin{lemma} Suppose $f = p^{c} f_0$, with $f_0$ in each of the cases enumerated above.  Denote by $\Lambda_f$ the rank two $\mathcal{O}$-lattice corresponding to $f$.  The content $c(\Lambda')$ of the index $q$ sublattices of $\Lambda_f$ are given as follows, in each of the various cases:
\begin{enumerate}
\item $f_0$ irreducible mod $p$: all $q+1$ index $q$ sublattices $\Lambda'$ have $c(\Lambda') = c -1$;
\item $f_0 = \ell q$ mod $p$: $1$ sublattice $\Lambda' = \Lambda_{\ell}$ has $c(\Lambda_\ell) = c$; the other $q$ sublattices have $c(\Lambda') = c-1$;
\item $f_0 = \ell_1 \ell_2 \ell_3$ mod $p$: Three lattices $\Lambda' = \Lambda_{\ell_i}$, $i = 1,2,3$, have $c(\Lambda_{\ell_i}) = c$; the other $q-2$ sublattices have $c(\Lambda') = c-1$.
\item $f_0 = \ell_1^2 \ell_2$ mod $p$: $1$ lattice $\Lambda' = \Lambda_{\ell_2}$ has $c(\Lambda_{\ell_2}) = c$; one lattice $\Lambda_{\ell_1}$ has $c(\Lambda_{\ell_1}) = c+1$; the other $q-1$ lattices have $c(\Lambda') = c-1$.
\item $f_0 = \alpha \ell^3$ mod $p$: $1$ lattice $\Lambda_{\ell}$ has $c(\Lambda_{\ell}) = c+2$; the other $q$ lattices have $c(\Lambda') = c-1$.\end{enumerate}
\end{lemma}
\begin{proof} This is straightforward.\end{proof}

From the lemma, and the fact that $T(p^{-1})= p^{-1}T(p)$, it is straightforward to evaluate $M_h(z)$.  For example, if $c = c(h) \geq 2$, then there are no ``boundary conditions'', and one obtains the following for each of the $5$ possible cases.  Set $v = \val(\det(h))$ and $g(z) = q^2 z^{v-c+1}(1-z^{c+2}) +z^{v-c-1}(1-z^{c})+ qz^{v-c}(1-z^{c+1})$.  Then 
\begin{align*} f_0 \text{ irred} &: g(z) + q(q+1) z^{v-c+2}(1-z^{c}) + (q+1)z^{v-c+1}(1-z^{c-1})\\
f_0 = \ell q &: g(z) + q\left(z^{v-c+1}(1-z^{c+1}) + q z^{v-c+2}(1-z^{c})\right) + \left(z^{v-c}(1-z^{c}) + q z^{v-c+1}(1-z^{c-1})\right)\\
f_0 = \ell_1 \ell_2 \ell_3 &: g(z) +q\left(3 z^{v-c+1}(1-z^{c+1}) +(q-2)z^{v-c+2}(1-z^{c})\right) \\ &\;\;\; + \left(3 z^{v-c}(1-z^{c}) + (q-2)z^{v-c+1}(1-z^{c-1})\right) \\
f_0 = \ell_1^2 \ell_2 &: g(z) + q\left(z^{v-c+1}(1-z^{c+1}) + z^{v-c}(1-z^{c+2})+(q-1)z^{v-c+2}(1-z^{c})\right) \\ &\;\;\; + \left(z^{v-c}(1-z^{c})+z^{v-c-1}(1-z^{c+1}) + (q-1)z^{v-c+1}(1-z^{c-1})\right)\\
f_0 = \alpha \ell^3 &: g(z) + q\left(z^{v-c-1}(1-z^{c+3})+qz^{v-c+2}(1-z^{c})\right) + \left(z^{v-c+2}(1-z^{c+2})+qz^{v-c+1}(1-z^{c-1})\right).
\end{align*}

From this table, one obtains that $B_0(z)P_h(z) -z^2 M_h(z)$ is as follows, in each of the above cases, when $c \geq 2$:
\[\begin{array}{c|l} f_0 \text{ irred} &  z^{v-c}\left(1+qz-z^3-qz^4\right) = z^{v-c}(1+qz)(1-z^3)\\
f_0 = \ell q & z^{v-c}\left(1+qz-z^2-qz^3\right) = z^{v-c}(1+qz)(1-z^2) \\
f_0 = \ell_1 \ell_2 \ell_3 & z^{v-c}\left(1+qz-3z^2+(2-3q)z^3+2q z^4\right) = z^{v-c}(1+qz)(1-z)^2(1+2z) \\
f_0 = \ell_1^2 \ell_2 &  z^{v-c}\left(1+(q-1)z-(q+1)z^2 -(q-1)z^3+qz^4\right) = z^{v-c}(1+qz)(1-z)(1-z^2)\\
f_0 = \alpha \ell^3 & 0.
\end{array}.\]

We also require the following lemma.
\begin{lemma} One has $\val(h) = c(h)$ if and only if $f_0 \neq \ell^3$ modulo $p$. \end{lemma}
\begin{proof} First suppose that $\val(h) = c(h)$.  This means $p \nmid h_0$, and thus we can suppose (making an appropriate change of variables) that $h_0 = \diag(p^r,1)$ with $r \geq 0$.  And if $r =0$, $f_0 = f_{max}$, so we are done.  Thus we can suppose $r \geq 1$.  Then $f_0(x,y) = \frac{1}{p^r} f_{max}(x,p^r y) = \frac{a}{p^{r}}x^3 + bx^2 y + cp^r xy^2 + p^{2r}d y^3$.  Hence $p^{r}|a$; set $a' = a/p^r$. Then $f_0 = a'x^3 + bx^2 y = x^2(a'x + by)$ modulo $p$.  To check this direction of the lemma, we thus need to see that $b$ is not $0$ modulo $p$.  But if $b =0 $ mod $p$, then $f_{max} = cxy^2 +dy^3$ mod $p$, which cannot happen by assumption.  Hence this direction of the lemma is verified.

Conversely, suppose $v(h) \neq c(h)$, so that we may assume $h_0 = p^{t} \diag(p^r,1)$ with $t \geq 1$.  Now, one always has that $f_{max} \cdot \left(p^{r} \diag(p^r,1)\right)$ is integral, thus $1 \leq t \leq r$.  Hence
\[f_0 = f_{max} \cdot h_0 = \frac{1}{p^{r-t}}(a x^3 + bp^r x^2 y + c p^{2r} xy^2 + d p^{3r} y^3) = a'x^3 + bp^{t} x^2 y + c p^{r+t} xy^2 + dp^{2r+t}y^3.\]
Since $t \geq 1$, $f_0 = a' x^3 = \alpha \ell^3$ modulo $p$, as desired.\end{proof}

Combining the previous calculations, we obtain the following proposition.  Note that $1-z + N'(T,h_0)z = 1 + (N(T_{max})-\epsilon(h_0))z$.
\begin{proposition}\label{prop:CRident} Let the notation be as above.  Then 
\[(1+qz)^{-1}L(E,s)\left(B_0(z)P_h(z) - z^2 M_h(z)\right) = \charf(\val(h) = c(h))z^{v-c}(1 + (N(f_{max})-\epsilon(h_0))z).\]
\end{proposition}
\begin{proof} We leave the proof of this identity in the cases $c(h)=0$ and $c(h)=1$ to reader.  Thus, suppose $c = c(h) \geq 2$.

Now, in all but the case that $f_0 = \ell_1^2 \ell_2$ mod $p$, the proposition follows immediately from the above calculations.  In case $f_0 = \ell_1^2 \ell_2$ mod $p$, then we can have either $E= F \times K$ or $E= F \times F \times F$.  In the first case, the left-hand side becomes $z^{v-c}(1-z)$.  In the second case, the left-hand side becomes $z^{v-c}(1+z)$.  In both cases, because $\ell_1^2 \ell_2$ does not correspond to the maximal ring, $\epsilon(h_0)=2$.  Thus, when $E = F \times K$, $1 + (N(T_{max})-\epsilon(h_0))z = 1+ (1-2)z = 1-z$, and when $E = F \times F \times F$, $1 + (N(f_{max})-\epsilon(h_0))z = 1+ (3-2)z = 1+z$.  Thus we get agreement in both cases.  This completes the proof of the proposition.\end{proof}

We have now proved (\ref{eqn:CRident1}), and with it, Theorem \ref{thm:unram}.

\subsection{The Dirichlet series} By Proposition \ref{prop:CRident}, one obtains that
\[\frac{L(\pi_p,Std,s-1)}{L(E,s)\zeta_F(2s-2)} = \frac{1}{1-z} \left(\sum_{[h]}{L(m(h) v) |\det(h)|^{-2} P_h(z)}\right).\]
In other words, one obtains
\[ \frac{L(\pi_p,Std,s-1)}{L(E,s)\zeta_F(2s-2)} = \sum_{[h], [\lambda]}{L(m(h) v) |\det(h)|^{s-2} |\lambda|^{-s} A_0(\lambda,h)}.\]

Make the change of variables $h = \lambda x$.  Then we finally obtain
\[\frac{L(\pi_p,Std,s-1)}{L(E,s)\zeta_F(2s-2)} = \sum_{[x] \text{ cubic ring},[\lambda]}{L(m(\lambda x) v) |\lambda^2 \det(x)|^{-2} |\lambda \det(x)|^{s}}.\]
This is the desired Dirichlet series for the standard $L$-function on $G_2$.  Converting this local expression to the ``classical'' notation of section \ref{sec:MFs} proves Theorem \ref{thm:DS}.

\section{Archimedean zeta integral}\label{sec:arch} In this final section, we combine our work from the previous two sections, and analyze the archimedean zeta integral that arises from the global integral $I(\varphi,\Phi,s)$ when $\varphi$ is a modular form of even weight $n$.  Throughout this section, the ground field $F$ is the field of real numbers $\R$.  We are able to relate this integral to products of gamma functions $\Gamma(s)$ and a certain explicit integral involving binary cubic forms, of the type considered in \cite{shintani}.
  
Suppose $\Phi: \g_E \rightarrow \V_n^\vee$ is a Schwartz function satisfying $\Phi(r v) = r \cdot \Phi(v)$ for all $v \in \g_E$ and $r\in K_E$, the maximal compact subgroup of $G_E(\R)$.  Denote by $\{\cdot,\cdot\}_{K}$ the symmetric $K_E$-invariant pairing on $\V_n$ or $\V_n^\vee$.  In this section, we analyze the integral
\[I(s,\Phi) = \int_{N^{0,E}(\R) \backslash \GL_1(\R) \times G_2(\R)}{|t|^{s} \{\Phi(t g^{-1} \widetilde{v_E}), \W_\chi(g)\}_{K} \,dg}\]
for a particular choice of $\Phi$.

\subsection{Setup} We now record the setup in greater detail.  Denote by $pr_{K}: \g_E\otimes \C \rightarrow \sl_2$ the $K_E$-equivariant surjection to the Lie algebra of the long root $\SU(2) \subseteq K_E$.  As in section \ref{sec:genWhit}, we can identify this $\sl_2$ with $Sym^2(V_2)$, and thus there is a $n^{th}$ power map $\sl_2 \rightarrow \V_n \simeq \V_n^\vee$, which we write as $X \mapsto X^n$ for $X \in \sl_2$.  We write $\{e_{\ell},h_{\ell},f_{\ell}\}$ for the $\sl_2$-triple of this $\sl_2$ from \cite{pollackQDS}.  For $X \in \g_E$, we write $||X|| = B_{\theta}(X,X)^{1/2}$.  Here $B_{\theta}(X,Y) = -B(X,\theta(Y))$ is a positive-definite symmetric pairing on $\g_E$.  See \cite[section 4]{pollackQDS} for the Cartan involution $\theta$ and invariant form $B(X,Y)$ on $\g_E$.  With these preliminaries, we define $\Phi(v) = pr_K(v)^n e^{-||v||^2}$. 

From this definition, one obtains
\begin{align*} f(\gamma_0 g,\Phi,s) &= \int_{\GL_1(\R)}{|t|^{s}\Phi(t g^{-1} \widetilde{v_E})\,dt} \\ &= \int_{\GL_1(\R)}{|t|^{s}pr_K(t g^{-1} \widetilde{v_E})^{n}e^{-t^2 ||g^{-1} \widetilde{v_E}||^2}\,dt} \\ &= \Gamma\left(\frac{s+n}{2}\right) pr_{K}(g^{-1} \widetilde{v_E})^{n} ||g^{-1}\widetilde{v_E}||^{-(s+n)}.\end{align*}
Here we have used that $n$ is even.

Denote by $\chi: N(\R) \rightarrow \C^\times$ the character given by $\chi(n) = e^{i \langle v_E, n \rangle}$.  For $z \in \mathcal{H}_1^{\pm}$, set
\[r_0(z) = (1,-z,z^2,-z^3) = E_{12}-z v_1 + z^2 \delta_3 - z^3 E_{23}\]
in $N^{ab} \otimes \C$.  From Theorem \ref{thm:FE}, we have $\W_\chi(m) = \sum_{-n \leq v \leq n}{\W_\chi^v(m) \frac{x^{n+v}y^{n-v}}{(n+v)!(n-v)!}}$ with
\begin{equation}\label{eqn:Wchim}\W_\chi^v(m) = \det(m)^n |\det(m)| \left(\frac{|\langle v_E, m r_0(i)\rangle|}{\langle v_E, m r_0(i)\rangle}\right)^{v}K_v\left(|\langle v_E, m r_0(i)\rangle|\right).\end{equation}

For $m \in M(\R)$ and $n \in N(\R)$ in the Levi and unipotent radical of the Heisenberg parabolic of $G_2(\R)$, we set
\[x(n,m) = (nm)^{-1} \widetilde{v_E} = m^{-1} \widetilde{v_E} + \det(m)^{-1}\langle v_E, \overline{n} \rangle E_{13} \in \g_E.\]

\begin{lemma} With the notation above, one has the following:
\begin{enumerate}
\item $pr_K(x(n,m)) = \frac{1}{4}\det(m)^{-1}\left(\langle \widetilde{v_E}, m r_0(-i)\rangle e_{\ell} - \langle \widetilde{v_E}, m r_0(i) \rangle f_{\ell} - i \langle \widetilde{v_E}, \overline{n} \rangle h_{\ell}\right)$;
\item $||x(n,m)||^{2} = |\det(m)|^{-2}\left( |\langle \widetilde{v_E}, m r_0(i) \rangle|^{2} + |\langle \widetilde{v_E}, \overline{n}\rangle|^{2}\right)$.\end{enumerate}
\end{lemma}
\begin{proof} Denote by $B$ the scalar multiple of the Killing form on $\g_E$ normalized so that $B(e_{\ell},f_{\ell}) = 1$.  Then for $x \in \g_E$, one has
\[pr_{K}(x) = B(x,f_{\ell})e_{\ell} + \frac{1}{2} B(x,h_{\ell}) h_{\ell} + B(x,e_{\ell}) f_{\ell}.\]
From this, one obtains the first statement of the lemma.  For the second statement, one uses the identity $|| m^{-1} \widetilde{v_E}||^2 = |\langle m^{-1}\widetilde{v_E},  r_0(i) \rangle|^{2}$ that holds because $m^{-1} \widetilde{v_E}$ is rank one.\end{proof}

Putting the various pieces together, and applying the Iwasawa decomposition,  we arrive at
\[I(s,\Phi) = \Gamma\left(\frac{s+n}{2}\right) \int_{\GL_2(\R)}\int_{(N^{0,E}\backslash N)(\R)}{\frac{|det(m)|^{-3} e^{i \langle v_E, \overline{n}\rangle}}{||x(n,m)||^{(s+n)}} \{pr_K(x(n,m))^{n}, \W_\chi(m) \}_K\,dn\,dm}.\]
In the rest of this section we will compute this integral.

\subsection{The computation} To go further, we need an expression for $\{pr_K(x(n,m))^{n}, \W_\chi(m) \}_K$.  To ease notation, set $\alpha = \langle \widetilde{v_E}, m r_0(i)\rangle$ and $\beta = \langle \widetilde{v_E}, \overline{n} \rangle$.  The Lie algebra $\sl_2$ of the complexified long root $\SU(2) \subseteq K$ is identified with $Sym^2(V_2)$ via the map $e_{\ell} \mapsto x^2$, $h_{\ell} \mapsto -2xy$ and $f_{\ell} \mapsto -y^2$.  Thus $x(n,m) = \frac{1}{4}\det(m)^{-1}(\alpha^* x^2 + 2i \beta xy + \alpha y^2)$.  We normalize the pairing $\{ \cdot, \cdot \}_K$ so that $\{x^{n-v}y^{n+v}, x^{n+v}y^{n-v}\}_K = (-1)^{n+v} (n+v)! (n-v)!$.  Plugging in (\ref{eqn:Wchim}), with these normalizations one obtains
\begin{align*}\{pr_K(x(n,m))^{n}, \W_\chi(m)\}_K &= \frac{n!}{2^{2n}}|\det(m)| \sum_{j_1+j_2 +j_3 = n}(-1)^{j_3-j_1}\frac{(\alpha^*)^{j_1}}{(j_1)!}\frac{(2i\beta)^{j_2}}{(j_2)!} \frac{\alpha^{j_3}}{(j_3)!} \\ &\qquad \qquad \qquad \qquad \qquad \times \left(\frac{|\alpha|}{\alpha}\right)^{j_3-j_1}K_{j_3-j_1}(|\alpha|)\\ &= \frac{|\det(m)|}{2^{2n}} \sum_{j_1+j_2 +j_3 = n}{(-1)^{j_3-j_1}\frac{n!}{(j_1)!(j_2)!(j_3)!} (2i\beta)^{j_2}|\alpha|^{n-j_2} K_{j_3-j_1}(|\alpha|)}.\end{align*}

Iterating the identity $K_{n+1}(x) + K_{n-1}(x) = -2K_n'(x)$, one obtains
\[\sum_{j_1,j_3 \geq 0, j_1+j_3 = n-j_2}{\frac{(n-j_2)!}{(j_1)!(j_3)!}K_{j_3-j_1}(|\alpha|)} = (-2)^{n-j_2}K_0^{(n-j_2)}(|\alpha|).\]
Thus
\[ \{pr_K(x(n,m))^{n}, \W_\chi(m)\}_K = \frac{|\det(m)|}{2^n}\sum_{j}{\binom{n}{j}(i\beta)^j |\alpha|^{n-j} K_0^{(n-j)}(|\alpha|)}.\]
Therefore
\begin{align*} I(s,\Phi) = 2^{-n} \Gamma\left(\frac{s+n}{2}\right) & \int_{\GL_2(\R) \times (N_{0,E}\backslash N)(\R)} \frac{|\det(m)|^{s+n-2} e^{i\beta}}{(|\alpha|^2+|\beta|^2)^{\frac{s+n}{2}}} \\  & \;\;\; \times \left(\sum_{j}{\binom{n}{j}(i\beta)^{n-j} |\alpha|^{j} K_0^{(j)}(|\alpha|)}\right)\,dn\,dm.\end{align*}

For $r, y > 0$, one has
\[\frac{\Gamma(s)}{2\Gamma(1/2)} \int_{\R}{\frac{e^{irx}}{(x^2+y^2)^s}\,dx} = \left(\frac{r}{2y}\right)^{s-1/2}K_{s-1/2}(ry)\]
and thus
\[\frac{\Gamma(s)}{2\Gamma(1/2)} \int_{\R}{\frac{(ix)^k e^{irx}}{(x^2+y^2)^s}\,dx} = \partial_r^{k}\left(\left(\frac{r}{2y}\right)^{s-1/2}K_{s-1/2}(ry)\right).\]
We obtain that
\[\Gamma\left(\frac{s+n}{2}\right) \left(\int_{\R}{\frac{e^{i\beta} (i\beta)^j}{(|\alpha|^2+|\beta|^2)^{(s+n)/2}}\,d\beta}\right) |\alpha|^{n-j}K_0^{(n-j)}(|\alpha|)\]
differs from 
\[ \left\{\partial_r^{j}\left( \left(\frac{r}{2|\alpha|}\right)^{(s+n-1)/2}K_{(s+n-1)/2}(r|\alpha|) \right) \partial_{r}^{n-j}\left(K_0(r|\alpha|)\right) \right\}|_{r=1}\]
by a nonzero absolute constant. Equivalently,
\[\Gamma\left(\frac{s+n}{2}\right) \left(\int_{\R}{\frac{e^{i\beta} (i\beta)^j}{(|\alpha|^2+|\beta|^2)^{(s+n)/2}}\,d\beta}\right) |\alpha|^{n-j}K_0^{(n-j)}(|\alpha|)\]
is equal  to
\[2^{-s/2}|\alpha|^{1-s}\left\{\partial_y^{n-j}\left(K_0(y)\right)\partial_y^{j}\left(y^{(s+n-1)/2}K_{(s+n-1)/2}(y)\right)\right\}|_{y=|\alpha|}\]
up to a nonzero absolute constant.

We have therefore proved
\begin{equation}\label{eqn:niceForm1}I(s,\Phi) \stackrel{\cdot}{=} 2^{-s/2}\int_{\GL_2(\R)}{|\det(m)|^{s+n-2} |\alpha|^{1-s}\partial_y^{n}\left(y^{(s+n-1)/2}K_{(s+n-1)/2}(y)K_0(y)\right)|_{y =|\alpha|}\,dm}\end{equation}
where $\stackrel{\cdot}{=}$ means that the two sides are equal up to an element of $\C^\times$, and $\alpha = \langle \widetilde{v_E}, m r_0(i)\rangle$.  Recall $w_0 = \mm{-1}{}{}{1}$.  Because $w_0 r_0(i) = r_0(-i) =r_0(i)^*$ and the integrand in (\ref{eqn:niceForm1}) only depends on the absolute value of $|\alpha|$, this integral can be replaced by an integral over $\GL_2(\R)^{0}$.  Thus we can split the integral in (\ref{eqn:niceForm1}) into an integral over $\SL_2(\R)$ and $\R^\times_{>0}$.  

For $m=g w \in \GL_2(\R)$, with $g \in \SL_2(\R)$ and $w \in \R^\times_{>0}$, set $\alpha_0 = \langle \widetilde{v_E}, gr_0(i) \rangle$ so that $|\alpha| = w|\alpha_0|$.  Thus
\begin{align*} I(s,\Phi) &\stackrel{\cdot}{=} 2^{-s/2}\int_{\SL_2(\R) \times \R^\times_{>0}}{w^{s+2n-3}|\alpha_0|^{1-s}\partial_y^{n}\left(y^{(s+n-1)/2}K_{(s+n-1)/2}(y)K_0(y)\right)|_{y =|\alpha|}\,dg\,d^{\times}w} \\ &= 2^{-s/2}\int_{\SL_2(\R) \times \R^\times_{>0}}{w^{s+2n-3}|\alpha_0|^{-(2s+2n-4)}\partial_y^{n}\left(y^{(s+n-1)/2}K_{(s+n-1)/2}(y)K_0(y)\right)|_{y = w}\,dg\,d^{\times}w} \\ &= 2^{-s/2}\left(\int_{\SL_2(\R)}{|\alpha_0|^{-(2s+2n-4)}\,dg}\right) \\ &\qquad \qquad \times \left(\int_{0}^{\infty}{w^{s+2n-3} \partial_y^{n}\left(y^{(s+n-1)/2}K_{(s+n-1)/2}(y)K_0(y)\right)|_{y = w} \frac{dw}{w}}\right).\end{align*}

Integrating by parts on the second integral, we obtain
\[I(s,\Phi) \stackrel{\cdot}{=} 2^{-s/2}\frac{\Gamma(s+2n-3)}{\Gamma(s+n-3)}J(s+n-2) \int_{0}^{\infty}{w^{(3s+3n-7)/2}K_{(s+n-1)/2}(w)K_0(w)\,\frac{dw}{w}}\]
where $J(\nu) = \int_{\SL_2(\R)}{|\alpha_0|^{-2\nu}\,dg}$.  From the Mellin transform of the $K$-Bessel functions and an application of Barnes' lemma one has the general identity
\[\int_{0}^{\infty}{K_{\mu}(y) K_{\nu}(y) y^{s} \frac{dy}{y}} = 2^{s-3} \frac{ \Gamma\left(\frac{s+ \mu + \nu}{2}\right)\Gamma\left(\frac{s+ \mu - \nu}{2}\right)\Gamma\left(\frac{s- \mu + \nu}{2}\right)\Gamma\left(\frac{s - \mu -\nu}{2}\right)}{\Gamma(s)},\]
valid for $Re(s)$ large compared with $Re(\nu), Re(\mu)$.  Applying this, we get
\[I(s,\Phi) \stackrel{\cdot}{=} 2^{s}J(s+n-2)\frac{\Gamma(s+2n-3)\Gamma(s+n-2)^2\Gamma((s+n-3)/2)^2}{\Gamma(s+n-3)\Gamma((3s+3n-7)/2)}.\]

It remains to evaluate the function $J(\nu)$.  In coordinates, one has
\[J(\nu) = \int_{\mathcal{H}_1}{|p_{E}(z)|^{-2\nu}y^{3\nu} \frac{dx\,dy}{y^2}}\]
where $p_E$ is the cubic polynomial defining the cubic extension $E/\Q$. One can also manipulate as follows:
\begin{align*} \Gamma(\nu)J(\nu) &= \Gamma(\nu)\int_{\SL_2(\R)}{|\alpha_0|^{-2\nu}\,dg} = \int_{\SL_2(\R)}{\left(\int_{0}^{\infty}{t^\nu e^{-|\alpha_0|^2 t}\,\frac{dt}{t}}\right)\,dg} \\ &= \int_{\GL_2(\R)^{+}}{|\det(g)|^{\nu} e^{-|\alpha|^2}\,dg} \\ &\stackrel{\cdot}{=} \int_{V^{*}}{|q(v)|^{\nu/2} e^{-|\langle v,r_0(i)\rangle|^2} \frac{dV}{|q(v)|}}.\end{align*}
Here $q$ is the invariant quartic form on $W = Sym^{3}(V_2) \otimes \det^{-1}$; $V^{*}$ is the $\GL_2(\R)$-orbit that consists of the binary cubics with $3$ distinct roots; and $d V$ is the additive Haar measure on $W$, so that $\frac{dV}{|q(v)|}$ is $\GL_2(\R)$-invariant.

One can show that the function $v \mapsto e^{-|\langle v, r_0(i)\rangle|^2}$ is a Schwartz function on $V^{*}$.  It follows that this final integral if of the form considered in \cite{shintani}.  From \cite[page 162]{shintani}, one obtains
\[\Gamma(\nu)J(\nu) = h(\nu) \Gamma\left(\frac{\nu}{2}-\frac{1}{6}\right)\Gamma\left(\frac{\nu}{2}\right)^2\Gamma\left(\frac{\nu}{2}+\frac{1}{6}\right)\]
for some function $h(\nu)$ that is analytic in $\nu$.

Summing up, we have proved the following result.
\begin{theorem} Let the notation be as above.  
\begin{enumerate}
\item One has
\[I(s,\Phi) \stackrel{\cdot}{=} 2^{s} \left(\int_{V^*}{|q(v)|^{(s+n-2)/2}e^{-|\langle v, r_0(i)\rangle|^2}\frac{dV}{|q(v)|}}\right) \frac{\Gamma(s+2n-3)\Gamma(s+n-2)\Gamma((s+n-3)/2)^2}{\Gamma(s+n-3)\Gamma((3s+3n-7)/2)}.\]
\item There is an entire function $R_n(s)$ so that
\[I(s,\Phi) = R_n(s) \frac{\Gamma(s+2n-3)\Gamma(s+n-2)\Gamma(s+n-3)}{\Gamma((s+n-1)/2)}.\]
\end{enumerate}
\end{theorem}
\begin{proof} The first part we have already for proved.  The second follows from the first, \cite[page 162]{shintani}, and a ``cubic'' application of the Gauss multiplication formula for the gamma function.\end{proof}

Evaluating the integral $J(\nu)$ exactly remains an interesting problem.

\bibliography{G2_MFs_bib}
\bibliographystyle{amsalpha}
\end{document}